\documentclass{amsart}

\usepackage{amsmath}
\usepackage{amssymb}
\usepackage{MnSymbol}
\usepackage{dsfont}
\usepackage{epsfig}  		
\usepackage[colorlinks=false,pdfborder={0 0 0}]{hyperref}
\usepackage{eucal}
 \usepackage{lineno}
\usepackage{xcolor}

\newtheorem{mthm}{Theorem}
\renewcommand\themthm{\Alph{mthm}}
\newtheorem{mcor}[mthm]{Corollary}

\renewcommand\thefurstenberg

\newtheorem{thm}{Theorem}[section]
\newtheorem{prop}[thm]{Proposition}
\newtheorem{lem}[thm]{Lemma}
\newtheorem{cor}[thm]{Corollary}

\theoremstyle{definition}
\newtheorem{definition}[thm]{Definition}
\newtheorem{example}[thm]{Example}

\theoremstyle{remark}
\newtheorem*{remark}{Remark}

\numberwithin{equation}{section}

\DeclareMathAlphabet{\mathpzc}{OT1}{pzc}{m}{it}

\newcommand{\cf}{{cf$.$\,}}

\newcommand{\CC}{\mathds{C}}
\newcommand{\RR}{\mathds{R}}

\newcommand{\ZZ}{\mathds{Z}} 
\newcommand{\kk}{\mathds{k}}

\newcommand{\PP}{\mathds{P}}

\newcommand{\ad}{\mathrm{ad}}
\newcommand{\Ad}{\mathrm{Ad}}

\let\Re\undefined
\DeclareMathOperator{\Re}{\mathrm{Re}}
\let\Im\undefined
\DeclareMathOperator{\Im}{\mathrm{Im}}

\newcommand{\Iso}{\group{Iso}}

\newcommand{\GL}{\group{GL}}
\newcommand{\PGL}{\group{PGL}}

\renewcommand{\SS}{\group{S}}
\newcommand{\TT}{\group{T}}

\newcommand{\SO}{\group{SO}}

\newcommand{\OO}{\group{O}}

\newcommand{\Aut}{\group{Aut}}

\newcommand{\Euc}{\group{E}}
\newcommand{\Zen}{\mathrm{Z}}

\newcommand{\Inn}{\mathrm{Inn}}

\newcommand{\scl}{\text{\textsc{l}}}
\newcommand{\scr}{\text{\textsc{r}}}

\newcommand{\group}{\mathrm} 

\newcommand{\ra}[1]{\text{\textsf{\itshape #1}}} 

\newcommand{\ac}[1]{\overline{#1}^{\rm z}}

\renewcommand{\hat}{\widehat}
\renewcommand{\rho}{\varrho}
\renewcommand{\tilde}{\widetilde}
\renewcommand{\bar}{\overline}

\DeclareMathOperator{\codim}{\mathrm{codim}}

\renewcommand{\epsilon}{\varepsilon}

\newcommand{\zsp}{\mathbf{0}} 
\newcommand{\met}{\langle\cdot,\cdot\rangle}
\newcommand{\one}{\{e\}}
\newcommand{\centre}[1]{\frz({#1})}

\DeclareMathOperator{\im}{\mathrm{im}}

\DeclareMathOperator{\Span}{\mathrm{span}}

\DeclareMathOperator{\supp}{\mathrm{supp}}

\newcommand{\frg}{\mathfrak{g}}
\newcommand{\frh}{\mathfrak{h}}
\newcommand{\fra}{\mathfrak{a}}
\newcommand{\frb}{\mathfrak{b}}
\newcommand{\frc}{\mathfrak{c}}
\newcommand{\frn}{\mathfrak{n}}
\newcommand{\frz}{\mathfrak{z}}
\newcommand{\frm}{\mathfrak{m}}
\newcommand{\frw}{\mathfrak{w}}
\newcommand{\frk}{\mathfrak{k}}
\newcommand{\frj}{\mathfrak{j}}
\newcommand{\fri}{\mathfrak{i}}

\newcommand{\frf}{\mathfrak{f}}
\newcommand{\frd}{\mathfrak{d}}

\newcommand{\frs}{\mathfrak{s}}
\newcommand{\frl}{\mathfrak{l}}

\newcommand{\frq}{\mathfrak{q}}

\newcommand{\frr}{\mathfrak{r}}

\newcommand{\zen}{\mathfrak{z}}
\newcommand{\nor}{\mathfrak{n}}
\newcommand{\heis}{\mathfrak{h}}
\newcommand{\osc}{\mathfrak{osc}}
\newcommand{\gl}{\mathfrak{gl}}
\renewcommand{\sl}{\mathfrak{sl}}
\newcommand{\su}{\mathfrak{su}}
\newcommand{\so}{\mathfrak{so}}

\newcommand{\euc}{\mathfrak{e}}

\newcommand{\inv}{\mathrm{inv}}

\newcommand{\ab}{\mathfrak{ab}}

\newcommand{\p}{\mathsf{p}}

\newcommand{\g}{\mathrm{g}}
\newcommand{\I}{\mathrm{I}}

\newcommand{\gs}{\frg_{\rm s}}

\begin{document}


\title[Isometry Lie algebras of homogeneous spaces]{Isometry Lie algebras of indefinite homogeneous spaces of finite volume}

\author[Baues]{Oliver Baues}
\address{Oliver Baues, Department of Mathematics, Chemin du Mus\'ee 23, University of Fribourg, CH-1700 Fribourg,  Switzerland} 
\email{oliver.baues@unifr.ch}


\author[Globke]{Wolfgang Globke}
\address{Wolfgang Globke, Faculty of Mathematics, University of Vienna, Oskar-Morgenstern-Platz 1, 1090 Vienna, Austria}
\email{wolfgang.globke@univie.ac.at}

\author[Zeghib]{Abdelghani Zeghib}
\address{Abdelghani Zeghib, \'Ecole Normale Sup\'erieure de Lyon, 
Unit\'e de Math\'ematiques Pures et Appliqu\'ees,
46 All\'ee d'Italie,
69364 Lyon,
France}
\email{abdelghani.zeghib@ens-lyon.fr}



\begin{abstract}
Let $\frg$ be a real finite-dimensional Lie algebra equipped with
a symmetric bilinear form $\met$.
We assume that $\met$ is
nil-invariant. This  means that every nilpotent operator 
in the smallest algebraic Lie subalgebra of endomomorphims containing 
the adjoint representation of $\frg$ is an
infinitesimal isometry for $\met$.
Among these Lie algebras are the isometry Lie algebras of 
pseudo-Riemannian manifolds of finite volume.
We prove a strong invariance property for nil-invariant symmetric
bilinear forms, which states that the adjoint representations of
the solvable radical and all simple subalgebras of non-compact type
of $\frg$ act by infinitesimal isometries for $\met$.
Moreover, we study properties of the kernel
of $\met$ 
and the totally isotropic ideals in $\frg$ in relation to the index of $\met$.
Based on this, we derive a structure theorem and a classification for 
the isometry algebras of indefinite homogeneous spaces of finite
volume with metric index at most two.
Examples show that the theory becomes significantly more
complicated for index greater than two.
We apply our results to study
simply connected pseudo-Riemannian homogeneous spaces of
finite volume.
\end{abstract}

\date{\today}  
\maketitle

%
%

\tableofcontents


\section{Introduction and main results}
\label{sec:intro}

Let $\frg$ be a finite-dimensional Lie algebra equipped 
with a symmetric bilinear form $\met$. The pair is called a metric Lie algebra. 
Traditionally, the bilinear form $\met$  is called \emph{invariant} if 
the adjoint representation of $\frg$ acts by skew linear maps.  
We will call $\met$  \emph{nil-invariant}, if 
every nilpotent operator in the smallest algebraic Lie subalgebra of endomomorphims 
containing the adjoint representation of $\frg$ is a skew linear map.
This nil-invariance condition appears to be significantly weaker than the 
requirement that $\met$ is \emph{invariant}. 

Recall that the dimension of a 
maximal totally isotropic subspace is called the \emph{index} of a symmetric bilinear form, and that 
the form is called \emph{definite} if its index is zero. Since definite bilinear forms do not admit nilpotent 
skew maps, the condition of nil-invariance is less restrictive and therefore more interesting 
for metric Lie algebras with bilinear forms of higher index. 

In this paper, we mainly study  finite-dimensional real 
Lie algebras $\frg$ with a nil-invariant symmetric bilinear form. 
We will discuss the general properties of  these metric Lie algebras,  
compare them with Lie algebras with invariant 
symmetric bilinear form, and derive elements of a classification theory, 
which give a complete description for  low index, in particular,
in the situation of index less than three.
 \medskip



\paragraph{\emph{Nil-invariant bilinear forms and  
isometry Lie algebras}}
The motivation for this article mainly stems from 
the theory of geometric transformation groups and
automorphism groups of geometric structures. 

Namely, consider a Lie group $G$ 
acting by isometries on a pseudo-Riemannian manifold $(M,\g)$ 
of finite volume. Then at each point $p\in M$, the scalar product $\g_p$
naturally induces a symmetric bilinear form $\met_p$
on the Lie algebra $\frg$ of $G$. As we  show in Section \ref{sec:isometry_algebras} 
of this paper, the bilinear form $\met_p$ is nil-invariant on $\frg$. Note that, in general, 
 $\met_p$ will be degenerate, since the
subalgebra $\frh$ of $\frg$ tangent to the stabilizer $G_{p}$ of $p$ 
is contained  in its kernel.

Isometry groups of  {Lorentzian metrics} (where the scalar products $\g_{p}$ are of index one)  
have been studied intensely.  Results obtained by Adams and Stuck \cite{AS1} in the compact situation, 
and by Zeghib \cite{zeghib0} amount to a classification of the isometry Lie
algebras of Lorentzian manifolds of finite volume.

In these works it is used prominently that, for  Lorentzian finite volume manifolds, 
the scalar products $\met_p$  are  
invariant by the elements of the nilpotent radical of $\frg$, 
\cf  \cite[\S 4]{AS1}. The latter condition is closely related to
nil-invariance,  but it is also significantly less restrictive. The role played by 
the stronger nil-invariance condition seems to have gone unnoticed so far.

Aside from Lorentzian manifolds, the classification problem for isometry Lie algebras 
of finite volume geometric manifolds with metric $\g$ of  arbitrary  index  
appears to be much more difficult. \smallskip 

Some more specific results have been obtained in the context of 
\emph{homogeneous} pseudo-Riemannian manifolds. 
Here, $M$ can 
be described   
as a coset space $G/H$,   and 
any associated metric Lie algebra $(\frg, \met_{p})$  locally
determines $G$ and $H$, as well as the geometry of $M$.
These  pseudo-Riemannian manifolds  
are model spaces of  particular interest.

Based on \cite{zeghib0}, a structure theory for Lorentzian homogeneous spaces of finite volume
is given by Zeghib \cite{zeghib}. 

Pseudo-Riemannian homogeneous spaces of arbitrary index
were studied by Baues and Globke \cite{BG}  for solvable Lie groups $G$.  They  found 
that, for solvable $G$,  the  finite volume condition implies that the stabilizer $H$ is a lattice in $G$ 
and that the metric on $M$ is induced by a bi-invariant metric on 
$G$. Also it was  observed in \cite{BG}  that  the nil-invariance condition holds for the isometry 
Lie algebras of finite volume homogeneous spaces, where it appears as   
a direct consequence of the Borel density theorem. The main result  in \cite{BG} 
amounts to showing the surprising fact that  any nil-invariant symmetric bilinear form 
on a solvable Lie algebra  $\frg$ is, in fact, an invariant form
(a concise proof is also provided in Appendix \ref{sec:newproof}).
\medskip 


By studying metric Lie algebras with nil-invariant 
symmetric bilinear form, the present work aims to further understand the isometry Lie algebras of
pseudo-Riemannian manifolds of finite volume. We will 
derive a structure theory which allows to completely describe such 
algebras in index less  than three. In particular, this classification contains 
%
all local models for  pseudo-Riemannian homogeneous spaces of finite volume
of  index less than three. 

\subsection{Main results and structure of the paper}
In Section \ref{sec:isometry_algebras}, we prove that the orbit maps of isometric
actions of Lie groups on pseudo-Riemannian manifolds of finite
volume give rise to nil-invariant scalar products on their
tangent Lie algebras.

Some basic definitions and properties of metric Lie algebras are
reviewed in Section \ref{sec:scalar}.

In favourable cases, nil-invariance of $\met$ already implies invariance.
For solvable Lie algebras $\frg$, this is always the case,
as was first shown 
in \cite{BG}.
These results are briefly summarized in Section \ref{sec:solvable}.
In this section, we will also review the classification of solvable Lie algebras with
invariant scalar products of index one and two. 
Their properties will  be needed further on.  

\subsection*{Strong invariance properties} 
In Section \ref{sec:biinvariant} we begin our investigation of
nil-invariant symmetric bilinear forms $\met$ on arbitrary Lie algebras. 
For any Lie algebra $\frg$, we let 
 $$ \frg=(\frk\times\frs)\ltimes\frr$$ 
 denote a Levi decomposition of
$\frg$, 
where $\frk$ is semisimple of compact type,
$\frs$ is semisimple of non-compact type and $\frr$ is the
solvable radical of $\frg$. 
For this, recall that $\frk$ is called of compact 
type if the Killing form of $\frk$ is definite and that $\frs$ is of
non-compact type if it has no ideal of compact type.  
We also write $$ \gs=\frs\ltimes\frr . $$

Our first main result is a \emph{strong invariance property}  for
nil-invariant symmetric bilinear forms:

\begin{mthm}\label{mthm:invariance}
Let $\frg$ be a real finite-dimensional Lie algebra,
let $\met$ be a  nil-invariant symmetric bilinear form
on $\frg$ and  $\met_{\gs}$ the restriction of $\met$ to $\gs$.
Then: 
\begin{enumerate}
\item
$\met_{\gs\! }$
is invariant by  the adjoint action of $\frg$ on $\gs$.
\item
$\met$ is invariant by $\gs$. 
\end{enumerate}
\end{mthm}

Note that \emph{any} scalar product on a semisimple Lie algebra $\frk$
of compact type is already nil-invariant, without any further
invariance property required.
Therefore, Theorem \ref{mthm:invariance} is as strong as one can hope
for.
  
\begin{remark}
 We would like to point out that the proof of Theorem \ref{mthm:invariance} works for Lie algebras over any field of
 characteristic zero if the notion of subalgebra of compact type $\frk$
 is replaced by the appropriate notion of maximal anisotropic semisimple 
 subalgebra of $\frg$.  The latter condition is equivalent to the requirement that the Cartan 
 subalgebras of $\frk$ do not contain any elements split over the ground field.
\end{remark}

 We obtain the following striking corollary to Theorem \ref{mthm:invariance}, 
 or rather to its proof: 
 
 \begin{mcor}\label{mcor:invariance_complex}
Let $\frg$ be a finite-dimensional Lie algebra over the field of complex numbers and 
$\met$ a  nil-invariant symmetric bilinear form on $\frg$.
Then $\met$ is invariant. 
\end{mcor}
\smallskip 

For any nil-invariant symmetric bilinear form $\met$, it is important to consider its kernel  $$  \frg^{\perp}= \{ X \in \frg \mid  X \perp \frg \} \; , $$   also called the \emph{metric radical} of $\frg$.  If $\met$ is invariant, then $\frg^{\perp}$ is an ideal of $\frg$. If $\met$ is nil-invariant, then,  in general, $ \frg^{\perp}$ is not even a subalgebra of $\frg$. 
Nevertheless, a considerable simplification of the exposition 
may be obtained by restricting results to 
metric Lie algebras whose radical $\frg^\perp$ does 
not contain any non-trivial ideals of $\frg$. Such metric Lie algebras will be called \emph{effective}. 
This condition is, of course,  natural from the geometric motivation. Moreover, it is 
not a genuine restriction, since by dividing out the maximal ideal of $\frg$ contained in $\frg^\perp$,
one may pass from any metric Lie algebra to a quotient metric 
Lie algebra that is effective. \medskip 

Theorem  \ref{mthm:invariance} determines the properties of 
$ \frg^{\perp}$ significantly as is shown in the following: 

\begin{mcor}\label{mcor:invariance_radical}
Let $\frg$ be a finite-dimensional real Lie algebra with 
 a nil-in\-var\-iant symmetric bilinear form $\met$.
Assume that  the metric radical $\frg^{\perp}$ does not contain any non-trivial ideal of $\frg$. Let
$\centre{\gs}$ denote the center of $\gs$. 
Then 
$$    \frg^\perp \subseteq \;  \frk \ltimes \centre{\gs}  \;  \text{ and } \;  [\frg^\perp, \gs] \subseteq \,  \centre{\gs} \cap \frg^\perp. $$
\end{mcor}
 

The proof of  Corollary \ref{mcor:invariance_radical} can be found in Section \ref{sec:classificationLie}, which is at the technical heart of our paper.
In Subsection \ref{sec:transporters} 
we start out by studying the
totally isotropic ideals in $\frg$, and in particular properties of the
metric radical $\frg^\perp$. The main part of the proof of  Corollary \ref{mcor:invariance_radical} 
is given in 
Subsection \ref{sec:metric_radical}. We then also prove that if in addition $\met$ is $\frg^{\perp}$-invariant,
then
$
[\frg^{\perp}, \gs] = \zsp$. 


As the form $\met$ may be degenerate, it is useful to introduce 
its \emph{relative index}. By definition, this  is the index
of the induced scalar product on the vector space  $\frg/\frg^\perp$. The  relative index 
mostly determines the geometric and algebraic type of the bilinear form $\met$.

For effective metric Lie algebras with relative index  
$\ell \leq 2$, we  further strengthen Corollary
 \ref{mcor:invariance_radical} by showing that,  with this additional requirement, 
 $\frg^\perp$ does not  intersect $\gs$. This is formulated in Corollary \ref{cor:gperp_l2}.
%
 
%
\subsection*{Classifications for small index}

Section \ref{sec:classificationLie} culminates in Subsection \ref{sec:fsubmodules},  where we give an
analysis of the action of semisimple subalgebras on the solvable 
radical of $\frg$. This imposes strong restrictions on the structure of $\frg$ 
for small relative index.

The combined results are summarized in Section \ref{sec:classificationLie_sum},  leading  to the 
following general structure theorem for the case $\ell\leq 2$:

\begin{mthm}\label{mthm:smallindex}
Let $\frg$ be a real finite-dimensional Lie algebra with
nil-invariant symmetric bilinear form $\met$ of relative index
$\ell\leq 2$,
and assume that $\frg^\perp$ does not contain a non-trivial ideal of $\frg$.
Then:
\begin{enumerate}
\item
The Levi decomposition \eqref{eq:levi} of\, $\frg$ is a direct sum of ideals:   $\frg=\frk \times \frs \times \frr$.
\item 
$\frg^\perp$ is contained in $\frk \times \frz(\frr)$ 
and  $\frg^\perp \cap \frr  = \zsp$. 
\item
$\frs\perp(\frk\times\frr)$ and $\frk\perp[\frr,\frr]$.
\end{enumerate}
\end{mthm}

Examples in Section \ref{sec:index3} illustrate that
the statements in Theorem \ref{mthm:smallindex} may fail 
for relative index $\ell\geq 3$.
\\

We specialize Theorem \ref{mthm:smallindex} to obtain
classifications of the Lie algebras $\frg$ in the cases $\ell=1$
and $\ell=2$.
As follows from the discussion at the beginning, these theorems
also describe the structure of isometry Lie algebras of  pseudo-Riemannian homogeneous spaces of
finite volume with index one or two (real signatures of type $(n-1,1)$ or $(n-2,2)$, respectively). \medskip

Our first result concerns the Lorentzian case: 
\begin{mthm}\label{mthm:lorentzian}
Let $\frg$ be a Lie algebra with nil-in\-variant symmetric bilinear
form $\met$ of relative index $\ell=1$,
and assume that $\frg^\perp$ does not contain a non-trivial ideal
of $\frg$.
Then one of the following cases occurs: 
\begin{enumerate}
\item[(I)]
$\frg=\fra\times\frk$, where $\fra$ is abelian and either
semidefinite or Lorentzian.
\item[(II)] $\frg=\frr\times\frk$, where $\frr$ is  Lorentzian of oscillator type.  
\item[(III)] $\frg=\fra \times\frk\times\sl_2(\RR)$, where 
$\fra$ is abelian and definite,  $\sl_2(\RR)$ is Lorentzian. 
\end{enumerate}
\end{mthm}

This classification of isometry Lie algebras for finite volume homogeneous Loren\-tzian manifolds 
is contained in Zeghib's
 \cite[Th\'eor\`eme alg\'ebrique 1.11]{zeghib}, which  uses a somewhat different
approach in its proof. 
%
Moreover, the list in \cite{zeghib} contains two additional cases of metric Lie algebras (Heisenberg algebra and tangent algebra of the affine group, compare Example \ref{ex:heisenberg} of the present paper) that cannot appear as Lie algebras of transitive Lorentzian
isometry groups, since they do not satisfy the effectivity condition. According to \cite{zeghib}, models of all three types (I)-(III) actually occur as isometry Lie algebras of homogeneous spaces $G/H$, in which case $\frh = \frg^{\perp}$ is a subalgebra tangent to a closed subgroup $H$ of $G$. 
 \medskip 

The algebraic methods developed here  also lead to a complete understanding in the case of 
signature $(n-2,2)$: 

\begin{mthm}\label{mthm:index2}
Let $\frg$ be a Lie algebra with nil-invariant symmetric bilinear 
form $\met$ of relative index $\ell=2$,
and assume that $\frg^\perp$ does not contain a non-trivial ideal
of $\frg$.
Then one of the following cases occurs:
\begin{enumerate}
\item[(I)]
$\frg=\frr\times\frk$, where $\frr$ is one of the following:
\begin{enumerate}
\item
$\frr$ is abelian.
\item
$\frr$ is Lorentzian of oscillator type.
\item
$\frr$ is solvable but non-abelian with invariant scalar product
of index $2$.
\end{enumerate}
\item[(II)]
$\frg=\fra\times\frk\times\frs$. Here, $\fra$ is abelian, 
$\frs = \sl_2(\RR) \times \sl_2(\RR)$ with a non-degenerate
invariant scalar product of index $2$.
Moreover, $\fra$ is definite. 
\item[(III)]
$\frg =\frr\times\frk\times\sl_2(\RR)$,
where $\sl_2(\RR)$ is Lorentzian,
and $\frr$ is one of the following:
\begin{enumerate}
\item
$\frr$ is abelian and either semidefinite or Lorentzian.
\item
$\frr$ is Lorentzian of oscillator type.
\end{enumerate}
\end{enumerate}
\end{mthm}

For the definition of an oscillator algebra, see Example \ref{ex:oscillator1}.
The possibilities for $\frr$ in case (I-c) of
Theorem \ref{mthm:index2}
above are discussed in Section \ref{sec:solvable_invariant}. 
Note further that the orthogonality relations of Theorem \ref{mthm:smallindex} 
part (3) are always satisfied.

\begin{remark}
Theorem \ref{mthm:index2} contains no information which of the  possible  algebraic models
actually do occur as isometry Lie algebras of homogeneous spaces of index two. This question needs to be considered on another occasion. 
\end{remark} 

We apply our results to study the isometry groups of simply
connected homogeneous pseudo-Riemannian manifolds of finite volume.
D'Ambra \cite[Theorem 1.1]{dambra} showed that a simply connected
compact analytic Lorentzian manifold (not necessarily homogeneous)
has compact isometry group, and she also gave an example of a simply
connected compact analytic manifold of metric signature $(7,2)$ that
has a non-compact isometry group.

Here, we study homogeneous spaces for arbitrary metric signature.
The main result is the following theorem:

\begin{mthm}\label{mthm:geometric}
Let $M$ be a connected and simply connected pseudo-Riemannian
homogeneous space of finite volume, $G=\Iso(M)^\circ$,
and let $H$ be the stabilizer subgroup in $G$ of a point in $M$.
Let $G=KR$ be a Levi decomposition, where $R$ is the solvable radical
of $G$.
Then:
\begin{enumerate}
\item
$M$ is compact.
\item
$K$ is compact and acts transitively on $M$.
\item
$R$ is abelian.
Let $A$ be the maximal compact subgroup of $R$. Then $A=\Zen(G)^\circ$.
More explicitely, $R=A\times V$ where $V\cong\RR^n$ and $V^{K}=\zsp$.
\item
$H$ is connected.
If $\dim R>0$, then $H=(H\cap K) E$, where $E$ and $H\cap K$ are
normal subgroups in $H$, $(H\cap K)\cap E$ is finite,
and $E$ is the graph of a non-trivial homomorphism
$\varphi:R\to K$, where the restriction $\varphi|_A$ is injective.
\end{enumerate}
\end{mthm}

In Section \ref{sec:simplyconn} we give examples of isometry groups
of compact simply connected homogeneous $M$ with non-compact radical.
However, for metric index $1$ or $2$ the isometry
group of a simply connected $M$ is always compact:

\begin{mthm}\label{mthm:sc_index2}
The isometry group of any simply connected pseudo-Riemannian
homogeneous manifold of finite volume with metric index $\ell\leq2$
is compact.
\end{mthm}

As follows from Theorem \ref{mthm:geometric}, the isometry Lie algebra of a
simply connected pseudo-Riemannian homogeneous space of finite volume
has abelian radical.
This motivates a closer investigation of Lie algebras with abelian
radical that admit nil-invariant symmetric bilinear forms in
Section \ref{sec:abelianradical}.
In this direction, we prove:

\begin{mthm}\label{mthm:abelian_rad1}
Let $\frg$ be a Lie algebra whose solvable radical $\frr$ is
abelian.
Suppose $\frg$ is equipped with a nil-invariant symmetric bilinear
form $\met$ such that the metric radical $\frg^\perp$ of $\met$
does not contain a non-trivial ideal of $\frg$.
Let $\frk\times\frs$ be a Levi subalgebra of $\frg$, where $\frk$
is of compact type and $\frs$ has no simple factors of compact
type.
Then $\frg$ is an orthogonal direct product of ideals
\[
\frg = \frg_1 \times \frg_2 \times \frg_3,
\]
with
\[
\frg_1=\frk\ltimes\fra,
\quad
\frg_2=\frs_0,
\quad
\frg_3=\frs_1\ltimes\frs_1^*,
\]
where $\frr=\fra\times\frs_1^*$ and $\frs=\frs_0\times\frs_1$ are
orthogonal direct products,
and $\frg_3$ is a metric cotangent algebra.
The restrictions of $\met$ to $\frg_2$ and $\frg_3$ are invariant
and non-degenerate.
In particular, $\frg^\perp\subseteq\frg_1$.
\end{mthm}

For the definition of metric cotangent algebra, see Example \ref{ex:sl2sl2}.
We call an algebra $\frg_1=\frk\ltimes\fra$ with
$\frk$ semisimple of compact type and $\fra$ abelian a Lie algebra
of \emph{Euclidean type}.
By Theorem \ref{mthm:geometric}, isometry Lie algebras of compact
simply connected pseudo-Riemannian homogeneous spaces are
of Euclidean type.
However, not every Lie algebra of Euclidean type appears as the
isometry Lie algebra of a compact pseudo-Riemannian homogeneous
space.
In fact, this is the case for the Euclidean Lie algebras
$\euc_n=\so_n\ltimes\RR^n$ with $n\neq 3$.

\begin{mthm}\label{thm:noSOnRn2}
The Euclidean group $\Euc_n=\OO_n\ltimes\RR^n$, $n\neq 1,3$,
does not have
compact quotients with a pseudo-Riemannian metric such that
$\Euc_n$ acts isometrically and almost effectively.
\end{mthm}

Note that $\Euc_n$ acts transitively and effectively on compact
manifolds with finite fundamental group, as we remark
at the end of Section \ref{sec:abelianradical}.

\subsection*{Notations and conventions}

%
%
The identity element of a group $G$ is denoted by $e$.
We let $G^\circ$ denote the connected component
of the identity of $G$.

Let $H$ be a subgroup of a Lie group $G$.
We write
$\Ad_{\frg}(H)$ for the
adjoint representation of $H$ on the Lie algebra $\frg$ of $G$,
to distinguish it from the adjoint representation $\Ad(H)$ on
its own Lie algebra $\frh$.

If $V$ is a $G$-module, then we write
$V^G=\{v\in V\mid gv=v \text{ for all } g\in G\}$ for the
module of \emph{$G$-invariants}.
Similary, $V^\frg=\{v\in V\mid Xv=0 \text{ for all } X\in\frg\}$
for a $\frg$-module.

The centralizer and the normalizer of $\frh$ in $\frg$ are denoted
by $\zen_\frg(\frh)$ and $\nor_\frg(\frh)$, respectively.
The center of $\frg$ is denoted by $\zen(\frg)$.
We use similar notation for Lie groups.

If $\frg_1$ and $\frg_2$ are two Lie algebras, the notation
$\frg_1\times\frg_2$ denotes the direct product of Lie algebras.
The notations $\frg_1+\frg_2$ and $\frg_1\oplus\frg_2$ are used
 to indicate sums and direct sums of vector spaces.

The \emph{solvable radical} $\frr$ of $\frg$ is the maximal
solvable ideal of $\frg$.
The semisimple Lie algebra $\frf=\frg/\frr$ is a direct product 
$\frf=\frk\times\frs$ of Lie algebras,
where $\frk$ is a semisimple Lie algebra of
\emph{compact type}, meaning that the Killing form of $\frk$ is
definite, and $\frs$ is semisimple without factors of compact type.

For any linear operator $\varphi$, $\varphi = \varphi_{\rm ss} + \varphi_{\rm n}$ denotes its Jordan decomposition, where $ \varphi_{\rm ss}$ is semisimple, and $\varphi_{\rm n}$ is nilpotent. 
Further notation will be introduced in Section \ref{sec:scalar}.

\subsection*{Acknowledgements}
Wolfgang Globke was partially supported by the Australian Research Council grant {DE150101647} and the Austrian Science Foundation FWF grant I 3248.
He would also like to thank the Mathematical Institute of the
University of G\"ottingen,
where part of this work was carried out, for its hospitality
and support.


\section{Isometry Lie algebras} 
\label{sec:isometry_algebras}

Let $(M,\g)$ be a pseudo-Riemannian manifold of finite volume,
and let $$ G \subseteq \Iso(M,\g) $$ be a Lie  group of isometries of $M$.
Identify the Lie algebra $\frg$ of $G$ with a subalgebra of 
Killing vector fields on $(M,\g)$.
Let $\SS^2\frg^*$ denote the space of symmetric bilinear
forms on $\frg$, and let
\[
\Phi:M \to  \SS^2\frg^* ,\quad p \mapsto \Phi_{p}
\]
be the Gau\ss~map, where
\[
\Phi_{p}(X,Y) = \g_{p}(X_{p}, Y_{p}).
\]

\vspace{1ex}
The adjoint representation of $G$ on $\frg$ induces a representation
$\rho:G\to\GL(\SS^2\frg^*)$. 

\begin{thm}\label{thm:isom_is_nilinvariant}
Let $\ra{A}$ be the real Zariski closure of $\rho(G)$ in the group $\GL(\SS^2\frg^*)$. Let 
$p \in M$. Then the bilinear form $\Phi_{p}$ is invariant by all unipotent elements in 
$\ra{A}$. 
\end{thm}


\begin{proof}
Note that the above Gau\ss~map $\Phi$ is equivariant with respect to $\rho$, 
since $G$ acts by isometries on $M$.
The pseudo-Riemannian metric on $M$ defines a finite $G$-invariant measure
on $M$.

Since the claim clearly holds on totally isotropic
$G$-orbits, we may in the following assume that all orbits of $G$ are
non-isotropic, that is, $\Phi_p\neq0$ for all $p\in M$.

%

Put $V=\SS^2\frg^*$. For a subset $W\subseteq V\backslash\zsp$, let
$\bar{W}$ denote its image in the projective space $\PP(V)$.
Similarly, for subsets in $\GL(V)$  and their image in
the projective linear group $\PGL(V)$.

The finite $G$-invariant measure on $M$ induces a finite $G$-invariant
measure $\nu$ on the projective space
$\PP(V)$ with support $\supp\nu = \bar{\Phi(M)}\subset\PP(V)$.
Let $\PGL(V)_{\nu}$ denote the stabilizer of $\nu$ in the
projective linear group. This is a real algebraic subgroup of 
$\PGL(V)$, \cf\cite[Theorem 3.2.4]{zimmer}. 
Also, by construction, $\bar{\rho(G)}\subseteq \PGL(V)_{\nu}$.

There exist vector subspaces ${W}_1,\ldots,{W}_r$
of $V$ such that 
$\supp\nu \subseteq \bar{W} = \bar{W}_{1} \cup \ldots \cup \bar{W}_{r}$
and the quasi-linear subspace $\bar W$ is minimal with this property.
Note that the identity component of $\PGL(V)_{\nu}$ preserves
all $\bar{W}_{i}$, and by Furstenberg's Lemma \cite[Corollary 3.2.2]{zimmer},
its restriction to $\PGL(W_{i})$ has compact closure.

Since  $\PGL(V)_{\nu}$ is real algebraic, the image of $\ra{A}$ in
$\PGL(V)$ is contained in $\PGL(V)_{\nu}$. 
Choose $W_{i}$ such that $\Phi_{p} \in W_{i}$. Let $u \in \ra{A}$ be a 
unipotent element. Since the restriction of $u$ to 
$\PGL(W_{i})$ is unipotent and it is 
contained in a compact subset  of 
$\PGL(W_{i})$, it must be the identity of $\bar W_{i}$. 
This implies $u  \cdot \Phi_{p} = \Phi_{p}$.
\end{proof}

In terms of Definition \ref{def:nilinvariance} below, this
implies the following:

\begin{cor} \label{cor:isom_inilinvariant}
For $p \in M$,  let $\met_{p}$ denote the symmetric bilinear 
form induced on the Lie algebra $\frg$ of $G$ by pulling back $\g_{p}$ 
along the orbit map $g \mapsto g \cdot p$. Then 
$\met_{p}$ is nil-invariant and its kernel contains
the Lie algebra $\frg_p$ of the stabilizer $G_p$ of $p$ in $G$.
If $G$ acts transitively on $M$, 
then the kernel of $\met_{p}$ equals $\frg_p$.
\end{cor} 



\section{Metric Lie algebras}
\label{sec:scalar}

Let $\frg$ be a finite-dimensional real Lie algebra
with a symmetric bilinear form $\met$.
The pair $(\frg,\met)$ is called a \emph{metric Lie algebra}.\footnote{Some authors (e.g.~\cite{KO1}) use this term for Lie algebras
with an \emph{invariant} scalar product.}

Let $\frh$ be a subalgebra of $\frg$.
The restriction of $\met$ to $\frh$  will
be denoted by $\met_\frh$.
The form $\met$ is called
\emph{$\frh$-invariant} if
\begin{equation}
\langle \ad(X)Y_1,Y_2\rangle = - \langle Y_1,\ad(X)Y_2\rangle
\end{equation}
for all $X\in\frh$ and $Y_1,Y_2\in\frg$.
We define
\[
\inv(\frg,\met)=\{X\in\frg\mid \langle \ad(X)Y_1,Y_2\rangle=-\langle Y_1,\ad(X)Y_2\rangle \text{ for all } Y_1,Y_2\in\frg\}.
\]
This is the maximal subalgebra of $\frg$ under which $\met$
is invariant.
If $\met$ is $\frg$-invariant, we simply say $\met$ is \emph{invariant}.

The kernel of $\met$ is the subspace 
\[
\frg^\perp = \{X\in\frg \mid \langle X,Y\rangle=0 \text{ for all } Y\in\frg\} \, .
\]
It is also called the \emph{metric radical} for $(\frg, \met)$. It is an invariant subspace for the Lie  brackets with  elements of $\inv(\frg,\met)$, and,  if $\met$ is invariant, then
$\frg^\perp$ is an ideal in $\frg$.

%
\subsection{Nil-invariant bilinear forms}

Let $\ac{\Inn(\frg)}$ denote
the Zariski closure 
of the adjoint group $\Inn(\frg)$ 
in $\Aut(\frg)$. 

\begin{definition}\label{def:nilinvariance}
A symmetric bilinear form $\met$ on $\frg$ is called
\emph{nil-invariant}, if  for all $X_1,X_2\in\frg$, 
\begin{equation} \label{eq:nilinv}
\langle \varphi X_1, X_2\rangle = -\langle X_1, \varphi X_2\rangle,     
\end{equation}
for all nilpotent elements $\varphi$ of the
Lie algebra of $\ac{\Inn(\frg)}$.
\end{definition}

In particular, \eqref{eq:nilinv} holds for the nilpotent parts
$\varphi=\ad(Y)_{\rm n}$ of the Jordan decomposition of the adjoint
representation of any $Y\in\frg$.

%
\subsection{Index of symmetric bilinear forms}
\label{subsec:signature}
Let $\met$ be a symmetric bilinear form on a finite-dimensional vector space $V$.
An element $x\in V$ is called \emph{isotropic} if $\langle x,x\rangle=0$.
A subspace $W \subseteq V$ is called \emph{isotropic}, if there exists 
$x \in W$, $x \neq 0$, with  $\langle x,x\rangle=0$. $W$ is called
\emph{totally isotropic} if $W \subseteq W^{\perp}$. 

The dimension of a maximal totally isotropic subspace of $V$
is called the \emph{index} $\mu(V)$ of $V$. Set
\[
\ell(V) = \mu(V) -  \dim V^\perp,
\]
so that $\ell(V)$ is the index of the non-degenerate bilinear form
induced by $\met$ on $V/V^\perp$.
We call $\ell$ the \emph{relative index} of $V$ (or $\met$).

When there is no ambiguity about the space $V$, we simply write
$\mu=\mu(V)$ and  $\ell=\ell(V)$. We then say that $V$ is of \emph{index $\ell$ type}. 
In particular, for $\ell=1$, we say $V$ is of
\emph{Lorentzian type}.
We call $V$ \emph{Lorentzian} if $\mu=\ell=1$.

If $\met$ is non-degenerate, that is, if $V^\perp=\zsp$,
then we call $\met$ a \emph{scalar product} on $V$.
We say that the scalar product $\met$ is \emph{definite} if $\mu = 0$. 

Let $W\subseteq V$ be a vector subspace.
We say $W$ is \emph{definite}, \emph{Lorentzian}, of  relative index  $\ell(W)$ or of index  $\mu(W)$, respectively, if the restriction $\met_{W}$ is. Observe further that $\mu(W) \leq \mu(V)$ and $\ell(W) \leq \ell(V)$.

%

\subsection{Examples of metric Lie algebras}
\label{subsec:metric_lie_algebras}

\begin{example}\label{ex:abelian}
Consider $\RR^n$ with a scalar product $\met$ represented by the
matrix
$\left(\begin{smallmatrix}
\I_{n-s} & 0 \\
0 & -\I_s
\end{smallmatrix}\right)$,
where $s\leq n-s$. Then $\met$ has index $s$, and we write
$\RR^n_s$ for $(\RR^n,\met)$.
If we take $\RR^n$ to be an abelian Lie algebra,
together with $\met$ it becomes a metric Lie algebra
denoted by $\ab^n_s$.
\end{example}

The Heisenberg algebra occurs naturally in 
the construction of Lie algebras with invariant scalar products.

\begin{example}\label{ex:heisenberg} Let $(\cdot,\cdot)$ be a Hermitian form on $\CC^n$. 
Define the  \emph{Heisenberg algebra} $\heis_{2n+1}$
as the vector space $\CC^n\oplus\frz$, where $\frz=\Span\{Z\}$,
with Lie brackets defined by
\[
[X,Y]=\Im(X, Y) Z,
\]
for any $X,Y\in\CC^n$. 
Thus $\heis_{2n+1}$ is a real $2n+1$-dimensional two-step
nilpotent Lie algebra with one-dimensional center (as such
it is unique up to isomorphism of Lie algebras). Equip $\CC^n$
with the bilinear product $\met = \Re( \cdot  , \cdot  )$. Declaring
$\frz$ to be perpendicular to $\heis_{2n+1}$ turns $\heis_{2n+1}$
into a metric Lie algebra, whose relative index $\ell(\heis_{2n+1})$ is determined 
by the index of 
the Hermitian form. 
\end{example}

\begin{example}\label{ex:oscillator0}
Put $\frd=\Span\{J\}$. Define the $2n+2$-dimensional \emph{oscillator algebra}
$\osc$ as the semidirect product
\[
\osc  = \frd\ltimes\heis_{2n+1},
\]
where $J$ acts by multiplication with the imaginary unit on $\CC^n$. 
Given any metric on $\heis_{2n+1}$ as in Example \ref{ex:heisenberg}, an invariant scalar product $\met$ of 
index $\ell(\heis_{2n+1}) + 1$ on $\osc$ is obtained by requiring
$
\langle J,Z \rangle =1$ and 
$
\CC^n\perp J$.
\end{example}

%

Example \ref{ex:oscillator0} is an important special case of
the following construction:

\begin{example}\label{ex:general_oscillator}
Given $\psi\in\so_{n-s,s}$ define the \emph{oscillator algebra}
$$ \frg= \osc(\psi)$$ 
as follows.  On the vector space $\frg=\frd\oplus\ab^n_s\oplus\frj$
with $\frd=\Span\{D\}$, $\frj=\Span\{Z\}$ define a Lie product by declaring: 
\[
[D,X] = \psi(X),
\quad
[X,Y]=\langle[D,X],Y\rangle Z \; . 
\]
where,  $X,Y\in\ab^n_s$.
Next extend the
indefinite scalar product on $\ab^n_s$ to $\frg$ by
\[
\langle D,D\rangle = \langle Z,Z\rangle =0,
\quad
\langle D,Z\rangle = 1,
\quad
(\fra\oplus\frj)\perp\ab^n_s.
\]
Then $\met$ is an invariant scalar product of index $s+1$ 
on $\frg$.
The Lie algebra $\osc(\psi)$ is solvable. It is nilpotent if and
only if $\psi$ is nilpotent.  If $\psi$ is a $k$-step nilpotent operator, 
then $\frg$ is a $k$-step nilpotent algebra.
If $\psi$ is not zero, then the ideal $\frh=\ab_s^n\oplus\frj$ is
of Heisenberg type (that is, nilpotent with one-dimensional commutator
$[\frh,\frh]=\frj$).
%
\end{example}

\subsubsection{Invariant Lorentzian scalar products}
The main building blocks for metric Lie algebras with invariant Lorentzian scalar products
are obtained by: 

\begin{example}\label{ex:semisimple}
The Killing form on $\sl_2(\RR)$ is an invariant
Lorentzian scalar product.
In fact, all semisimple Lie algebras with an invariant Lorentzian
scalar product are products of $\sl_2(\RR)$ by simple factors of
compact type.
\end{example}


\begin{example}\label{ex:oscillator1}
For $\psi \in \so_{n}$, the oscillator algebra $\osc(\psi)$ is Lorentzian. 
We say that such a metric Lie algebra is Lorentzian of \emph{oscillator type}. 
\end{example}

\begin{remark}
Classification of Lie algebras with invariant Lorentzian
scalar products were derived by Medina \cite{medina}
and by Hilgert and Hofmann \cite{HH}. It can be deduced 
that algebras of oscillator type are the only non-abelian solvable
Lie algebras which admit an invariant Lorentzian scalar product.
This is also a direct consequence of the reduction theory
of solvable metric Lie algebras, see Section \ref{sec:solvable}.
\end{remark}



%



\section{Review of the solvable case}
\label{sec:solvable}

The first two authors studied nil-invariant symmetric bilinear forms on solvable Lie algebras in \cite{BG}.
The main result \cite[Theorem 1.2]{BG} is:

\begin{thm}\label{thm:nilinvariant}
Let $\frg$ be a solvable Lie algebra and  $\met$ a 
nil-invariant symmetric bilinear form on $\frg$.
Then $\langle\cdot,\cdot\rangle$ is invariant.
In particular, $\frg^\perp$ is an ideal in $\frg$.
\end{thm}

An important tool in the study of
(nil-)invariant products $\met$ on solvable $\frg$
is the reduction by a totally isotropic  ideal $\frj$
in $\frg$.
Since $\met$ is invariant,  $\frj^\perp$ is a subalgebra. Therefore, we can consider the quotient Lie algebra
\begin{equation}
\bar \frg = \frj^\perp / \, \frj.
\label{eq:reduction_j}
\end{equation}
Since $\frj$ is totally isotropic, $\bar \frg$ inherits a non-degenerate symmetric bilinear form from $\frj^{\perp}$ that is
(nil-)invariant as well.
The metric Lie algebra $(\bar \frg, \langle\cdot,\cdot\rangle)$ 
is called the \emph{reduction}  of $(\frg, \langle\cdot,\cdot\rangle)$ by $\frj$.
Reduction by $\frj$ decreases the index of $\met$.

Let $\frn$ be the nilradical of $\frg$.
The ideal
\begin{equation}
\frj_0=\zen(\frn)\cap[\frg,\frn]
\label{eq:j0}
\end{equation}
is a characteristic totally isotropic ideal in $\frg$,
whose orthogonal space $\frj_0^\perp$ is also an ideal in
$\frg$ and contains $\frj_0$ in its center.
Then $\frj_0=\zsp$ if and only if $\frg$ is
abelian.
In particular, $\frg$ is abelian if $\met$ is definite.
This implies \cite[Proposition 5.4]{BG}:

\begin{prop}\label{prop:completereduction}
Let $(\frg, \met)$ be a solvable metric 
Lie algebra with  nil-invariant symmetric bilinear form $\met$.
After a finite sequence of reductions 
with respect to totally isotropic and central ideals,  
$(\frg, \met)$ reduces to an abelian metric Lie algebra.
\end{prop}

The proposition is useful in particular to derive properties 
of  solvable metric Lie algebras of low index.

\subsection{Invariant scalar products of index $2$} \label{sec:solvable_invariant}


\begin{example} \label{ex:KO1} 
Let $\psi \in \so_{n,1}$. Then the oscillator algebra  $\osc(\psi)$ as defined in Example
\ref{ex:general_oscillator} is of index $\ell =2$. 
\end{example}

Let $\alpha^j=(\alpha^j_1,\ldots,\alpha^j_n)$, $j =  1,2$, denote
tuples of real numbers, and let us put  $\frd = \Span\{ D_{1}, D_{2}\}$,  
$\frj =\Span\{Z_1, Z_2\}$. Let $X_1,\ldots,X_n$, $Y_1,\ldots,Y_n$ 
be an orthonormal basis of $\fra= \ab^{2n}_0$,

\begin{example}\label{ex:KO2}
We define a metric Lie algebra $\frg=\osc(\alpha^1,\alpha^2)$ as follows.
The Lie product on $$ \frg = \frd \oplus\ab^{2n}_0\oplus \frj$$
 is given by the relations
\begin{equation} \label{eq:oscrel}
[X_i,Y_j]=\delta_{ij}(\alpha^1_i Z_1+\alpha^2_i Z_2),
\quad
[D_k,X_j]=\alpha^k_j Y_j,
\quad
[D_k,Y_j]=-\alpha^k_j X_j. 
\end{equation}
Define a scalar product $\met$ of index $2$ on $\frg$ 
by
\begin{equation} \label{eq:oscmetric}
\langle D_1,D_2\rangle=\langle Z_1,Z_2\rangle=0,
\quad
\langle D_i,Z_j\rangle=\delta_{ij},
\quad
D_i,Z_i\perp\fra \,  . 
\end{equation}
Then $\frg$ is a solvable Lie algebra with
invariant scalar product $\met$. 
Observe that 
$[\frg,\frg]=[\frg,\frn] \subseteq \fra \oplus\frj$, where $\frn$ is 
the nilradical of $\frg$. Then $\frn$  is at most two-step nilpotent, since
$[\frn, \frn] \subseteq \frj$. 
\end{example}

\begin{example}\label{ex:KO3}
We define a metric Lie algebra  $\frg=\osc_{1}(\alpha^1,\alpha^2)$ as follows. 
Consider $\fra = \ab^{2n+1}_0 = \Span\{W\} + \ab^{2n}_0$, where $W \perp \ab^{2n}_0$, and 
$\langle W, W \rangle =1$. 
A Lie product on 
$$
\frg=\frd \oplus\ab^{2n+1}_0 \oplus\frj
$$
is given by the
relations \eqref{eq:oscrel} and 
\begin{gather*}
[D_1,D_2]=W,
\quad
[D_1,W]=-Z_2,
\quad
[D_2,W]=Z_1. 
\end{gather*}
Define a scalar product $\met$ of index $2$ on $\frg$ using \eqref{eq:oscmetric}.  
Then $\frg$ is a solvable Lie algebra with
invariant scalar product $\met$.
Note if $n=0$, or $\alpha^{1} = \alpha^{2} = 0$ then $\frg$ is three-step nilpotent.
Otherwise, $[\frn, \frn ] \subseteq  \frj \subseteq \frz(\frg)$. 
\end{example}

The three families of Lie algebras in Examples \ref{ex:KO1}, \ref{ex:KO2}, \ref{ex:KO3} were
found by Kath and Olbrich \cite{KO1} to contain all
indecomposable non-simple metric Lie algebras with invariant
scalar product of index $2$. Thus we note:  

\begin{prop} \label{prop:solv_ind2}
Any solvable metric Lie algebra with invariant scalar product
of index $2$ is obtained by taking direct products of metric 
Lie algebras in Examples \ref{ex:oscillator1}, \ref{ex:KO1}  to \ref{ex:KO3} 
or abelian metric
Lie algebras.  
\end{prop} 

We use this to derive the following particular observation, which will
play an important role in Section \ref{sec:fsubmodules}.
An ideal in a metric Lie algebra $(\frg, \met)$ is called
\emph{characteristic} if it is preserved by every skew derivation of $\frg$. 

\begin{prop}\label{prop:index2algebras_q}
Let $\frg$ be a solvable Lie algebra with invariant
bilinear form $\met$ of index $\mu \leq 2$. 
Then $(\frg, \met)$ has a characteristic ideal $\frq$ that 
satisfies: 
\begin{enumerate}
 \item $\dim [\frq, \frq] \leq 2$.
 \item $\codim_{\frg} \frq \leq 2$. 
\end{enumerate}
\end{prop}
\begin{proof} It is easily checked that  the characteristic ideal $ \frq  = [\frg, \frg] + \frz(\frg)$ of $\frg$ 
satisfies (1) and (2)  for the Examples 
 \ref{ex:oscillator1}, \ref{ex:KO1}  to \ref{ex:KO3},
and for products of oscillators as in Example \ref{ex:oscillator1}. Hence, the proposition 
is satisfied for all invariant scalar products of index $\ell = \mu \leq 2$.  

Suppose now that $\met$ is degenerate and $\ell = 0$. Then $\frg/\frg^{\perp}$ inherits 
a definite invariant scalar product.  Hence, $\frg/\frg^{\perp}$ is abelian, 
and $[\frg, \frg] \subseteq \frg^{\perp}$. Since $\dim \frg^{\perp} \leq \mu \leq 2$, 
it follows that  $\frq= \frg$ has the required properties. 

Finally, suppose $\met$ is degenerate, $\ell = 1$. Then $\frg_{0}= \frg/\frg^{\perp}$ is Lorentzian. 
It follows that   $\frg_{0}$ admits a codimension one characteristic ideal $\frq_{0}  = [\frg_{0}, \frg_{0}] + \frz(\frg_{0})$, 
where $\dim [\frq_{0}, \frq_{0}]  \leq  1$. Thus the preimage  $\frq$ of $\frq_{0}$ in $\frg$ has the required properties. 
\end{proof} 

In the following $\frn$ denotes the nilradical of the Lie algebra $\frg$.

\begin{cor}\label{cor:index2algebras} 
Let $\frg$ be a solvable metric Lie algebra which admits an invariant
scalar product  of  index $\leq 2$. 
 If $\frg$ is not nilpotent 
then \[
\zen(\frg)\cap[\frn,\frn]=\zen(\frn)\cap[\frn,\frn]=[\frn,\frn] \; .
\]
\end{cor}


\section{Nil-invariant symmetric bilinear forms}
\label{sec:biinvariant}

Let $\frg$ be a finite-dimensional real Lie algebra with solvable radical $\frr$.
Let
\begin{equation}
\frg= (\frk \times \frs) \ltimes \frr
\label{eq:levi}
\end{equation}
be a Levi decomposition
, where $\frk$ is semisimple of compact type and
$\frs$ is semisimple without factors of compact type. 
Furthermore, we put 
\[
\gs  = \frs \ltimes  \frr  \; .
\]
Note that $\gs$ is a characteristic 
ideal of $\frg$.\\ 

The purpose of this section is to show:   

{
\renewcommand{\themthm}{\ref{mthm:invariance}}
\begin{mthm}
Let $\met$ be a  nil-invariant symmetric bilinear form on $\frg$,  and  let $\met_{\gs}$ denote the restriction of $\met$ to $\gs$.
Then: 
\begin{enumerate}
\item
$\met_{\gs\! }$
is invariant by  the adjoint action of $\frg$ on $\gs$.
\item
$\met$ is invariant by $\gs$. 
\end{enumerate}
\end{mthm}
}

The proof of Theorem \ref{mthm:invariance} begins with a few 
auxiliary results. 

\begin{lem}\label{lem:isotropic_s}
Let $\frs \subseteq \frg$ be a semisimple subalgebra of non-compact type. Then the subalgebra generated by all $X \in \frs$, such that  $\ad(X): \frg \to \frg$ is nilpotent, is $\frs$. 
\end{lem} 
\begin{proof} Call $X \in \frs$ nilpotent if  $\ad(X): \frs \to \frs$ is nilpotent. Since, for every representation of $\frs$, nilpotent elements are mapped to nilpotent operators, it is sufficient to prove the statement for $\frs = \frg$. 
So let $\frs_{0}$ be the subalgebra of $\frs$ generated by all nilpotent elements.
Since the set of all nilpotent elements is preserved by every automorphism
of $\frs$, it follows that
$\frs_{0}$ is an ideal. Therefore the semisimple Lie algebra $\frs_{1} = \frs/\frs_{0}$ does not contain any nilpotent elements. Let $\fra$ be a Cartan subalgebra of $\frs_{1}$, and $\fra_{\rm s}$ the subspace consisting of elements $X \in \fra$, where $\ad(X)$ is   
split semisimple (that is, diagonalizable over $\RR$). 
The weight spaces for the non-trivial roots of $\fra_{\rm s}$ consist of nilpotent elements of $\frs_{1}$.  Since, by construction,  $\frs_{1}$ has no 
nilpotent elements, this implies that $\fra$ has no elements split over $\RR$. 
This in turn implies that $\frs_{1}$ is of compact type (\cf Borel \cite[\S 24.6(c)]{borel}). 
By assumption,  $\frs$ is of non-compact type, so  $\frs_{1}$ must be trivial. 
\end{proof}

\begin{lem}\label{lem:sinvariant}
Let $\frn$ be the nilradical of $\frg$. Then 
$\met$ is invariant by $\frs \ltimes \frn$.
\end{lem}
\begin{proof}
Since $\met$ is nil-invariant, $\inv(\frg, \met)$ contains all $X$ such that the operator $\ad(X): \frg \to \frg$ is nilpotent. In particular,  $\frn$ is contained in  $\inv(\frg, \met)$. Since $\frs$ is of non-compact type, the subalgebra generated by all $X \in \frs$ with $\ad(X)$ nilpotent is $\frs$, see Lemma \ref{lem:isotropic_s} Therefore, also $\frs \subseteq  \inv(\frg, \met)$. 
\end{proof}




Recall that any derivation $\varphi$ of the solvable Lie algebra $\frr$  satisfies $\varphi(\frr) \subseteq \frn$
(Jacobson \cite[Theorem III.7]{jacobson}).
In particular, if $\varphi$ is semisimple, there exists a
decomposition  $\frr = \fra + \frn$ into vector subspaces,
where $\varphi(\fra) = 0$.
Similarly, for any subalgebra $\frh$ of $\frg$ acting reductively on $\frr$, we have $\frr = \frr^{\frh} + \frn$,
where $[\frh,\frr^{\frh}]=\zsp$.

\begin{lem}\label{lem:gs_orthog} Let $\frh$ be a subalgebra of\/ $\inv(\frg, \met)$, and
let $\frg^{\frh}$ be the maximal trivial submodule for the adjoint action of\/
$\frh$ on $\frg$.  Then $\frg^{\frh} \perp [\frh, \frg]$. Moreover, if $\frh$ is a semisimple 
subalgebra contained in 
$\frs$ then  
$[\frg^{\frh}, \frg] \perp \frh$. 
\end{lem}

\begin{proof}
%
Let $V \in  \frg^{\frh}$ and $X \in \frh$, $Y \in \frg$.
Then  $\langle V,[X,Y] \rangle = \langle [V,X],Y \rangle =0$.
Hence, $\frg^{\frh} \perp [\frh, \frg]$. 

Now assume $\frh$ is a semisimple subalgebra of $\frs$.
We may write  $Y = Y_1 + Y_2$, where  
$ Y_1 \in \frg^{\frs}$ and $Y_2 \in \frs \ltimes \frn$. 
Thus $[V, Y_{1}] \in \frg^{\frh}$. 
Since $\frh$ is also semisimple, $\frh = [\frh, \frh]$. 
Therefore, the first part of this lemma shows that $[V, Y_{1}] \perp \frh$.
By Lemma \ref{lem:sinvariant},  $Y_2  \in \inv(\frg, \met)$. 
Therefore, 
\[
\langle [V, Y_{2}], X \rangle   
= \langle V, [Y_2, X] \rangle =  \langle [X, V], Y_2 \rangle =0.
\] 
That is, $[V, Y_{2}] \perp \frh$ as well.
\end{proof}

%

\begin{proof}[Proof of Theorem \ref{mthm:invariance}, part (1)]
Since $\frs$ acts reductively on $\frg$, we have
$\frg = \gs + \frg^{\frs}$.
Therefore, by  Lemma \ref{lem:sinvariant}, 
it is enough to prove invariance of $\met_{\gs}$ under $\frg^{\frs}$. 

Let $X \in \frg^{\frs}$, $Y, Z \in \gs$. 
Decompose $Y = Y_\frs + Y_\frr$, $Z = Z_\frs + Z_\frr$, according
to the direct sum $\gs = \frs \oplus \frr$.
Using Lemma \ref{lem:gs_orthog}, we get 
$\langle [X, Y ], Z \rangle = \langle [X, Y_\frr  ], Z_\frr \rangle$.
By Theorem \ref{thm:nilinvariant}, the restriction  of
$\met$ to the solvable Lie algebra generated by $\frr$ and $X$   is invariant on that subalgebra.
Hence, $\langle [X, Y_\frr  ], Z_\frr \rangle
=  -\langle Y_\frr , [X, Z_\frr  ] \rangle =  -\langle Y  , [X, Z ] \rangle$. 
\end{proof}


\begin{lem}\label{lem:gk_orthog}
Let $\frf$ be a subalgebra of\/  $\frg$ and 
$\frg^{\frf}$ the maximal trivial submodule for the adjoint action of\/ $\frf$ on $\frg$. 
Then 
$[\frg^{\frf}, \gs] \perp \frf$. 
 In particular, 
$[\frg^{\frk}, \frg] \perp \frk$. 
\end{lem}
\begin{proof}
Let  $X\in \frg^{\frf}$, $Y \in \gs$ and $K \in \frf$. 
Since $\gs$ is an ideal in $\frg$,  we may write
\[
\gs = (\gs\cap\ker \ad(X)_{\rm ss})  +  \ad(X)\gs\,.
\]
Suppose first that $Y \in \ker \ad(X)_{\rm ss}$. Then we get
\[
\langle [X,Y], K\rangle = 
\langle \ad(X)_{\rm n}Y , K \rangle
=
-\langle Y, \ad(X)_{\rm n} K \rangle
= 0 \, .
\]
The latter term vanishes, since $K \in \ker \ad(X) \subseteq  \ker \ad(X)_{\rm n} $.

Next suppose $Y = \ad(X) Y'$, for some $Y' \in \gs$. Since 
$Y \in [\frg, \gs] \subseteq \frs \ltimes \frn$,  Lemma \ref{lem:sinvariant}
implies that $Y \in \inv(\frg,\met)$. 
Thus 
\[
\langle [X,Y], K\rangle =   \langle X, [Y, K ] \rangle =    \langle X, [ [X, Y'] , K ] \rangle =
-  \langle X, [ [Y', K] , X ] \rangle = 0 \; .
\] 
The latter term is zero, since $[Y',K] \in [ \gs, \frg] \subseteq \inv(\frg,\met)$,
and therefore $\ad([Y',K])$ is a skew-symmetric linear map. This shows $[\frg^{\frf}, \gs] \perp \frf$.
Finally, for the last statement observe that $[\frg^{\frk}, \frg] = [\frg^{\frk}, \gs]$. 
\end{proof}

\begin{proof}[Proof of Theorem \ref{mthm:invariance}, part (2)]
By Lemma \ref{lem:sinvariant},  $\met$ is invariant by $\frs + \frn$. Since $\gs = \frs + \frr^{\frk} + \frn$, 
to prove that  $\met$ is $\gs$-invariant, 
it suffices to show that $\ad(X)$ is skew for all $X \in \frr^\frk$.
By part (1), the restriction of $\ad(X)$ to the ideal $\gs$ is skew. 
Hence, it remains to show that $ \langle [X,Y],Z \rangle = - \langle Y, [X, Z]\rangle$,
where at least one of $Y,Z$, say $Y$, is in $\frk$. 
This is satisfied, since  \[ 0 = \langle [X,Y],Z \rangle =  -\langle Y ,[X,Z ]\rangle 
\; .   \]
Note that the right term is zero because of Lemma \ref{lem:gk_orthog}. 
 \end{proof}

Lemma \ref{lem:gs_orthog} and Lemma 
\ref{lem:gk_orthog} also imply:

\begin{cor}  \label{cor:ksperp} Let $\met$ be a nil-invariant symmetric bilinear form on
$\ \frg= (\frk \times \frs) \ltimes \frr $. Then
\begin{enumerate}
 \item $\frs \perp [\frk, \frg]$  and $\frk \perp [\frs, \frg]$. 
 \item The simple factors of $\frs$ are pairwise orthogonal.
\end{enumerate}
\end{cor}
 
\begin{example}[Nil-invariant products on semisimple Lie algebras]
Let
\[
\frg=\frk\times \frs
\]
be semisimple, where $\frk$ is an ideal of compact type and $\frs$ is 
of non-compact type. For any  nil-invariant bilinear form $\met$,  
\[
(\frg, \met) =  (\frk, \met_{\frk}) \times  (\frs, \met_{\frs})
\]  
decomposes as a direct product of metric Lie algebras,  where $\met_{\frs}$ is invariant.
\end{example}


%

%



\section{Totally isotropic ideals and metric radicals}
\label{sec:classificationLie}

Let $\frg$ be a finite-dimensional real Lie algebra with
nil-invariant symmetric bilinear form $\met$
and subalgebras $\frk$, $\frs$, $\frr$, $\gs$ as in Section
\ref{sec:biinvariant}. 
We let  $\ell$ denote the \emph{relative}  index of
$\met$ (which is the index of the
non-degenerate bilinear form induced by $\met$ on 
$\frg/\frg^\perp$).  

\subsection{Transporter algebras} \label{sec:transporters}

For any subspaces $U\subseteq V$ of $\frg$ and any subalgebra $\frq$ of
$\frg$, define
\[
\nor_{\frq}(V,U)=\{X\in\frq\mid [X,V]\subseteq U\}.
\]
Clearly, $\nor_{\frq}(V,U)$ is a subalgebra of $\frq$.
Also, 
$[\frq, V] \subseteq U$  if and only if  $\nor_{\frq}(V,U) = \frq$.
\medskip

Suppose $\frb \subseteq \gs$ is a totally isotropic ideal of $\frg$ contained in $\gs$. 
Then consider
$$ \frb_0  =  \frb \cap  \frg^\perp  \subseteq  \frb . $$ 
By Theorem \ref{mthm:invariance} part (2), $\met$ is invariant by $\gs$. Therefore, $\frb_0$ is an ideal of  
$ \gs$. \medskip
%

For any  subalgebra $\frq$ of $\frg$ define the
\emph{transporter subalgebra} for $\frb$ in $\frq$ as 
\begin{equation}
\nor_{\frq}(\frb,\frb_0)
=
\{X\in\frq \mid [X,\frb]\subseteq \frb_0\}.
\label{eq:transporter0}
\end{equation}


\begin{lem}[Transporter Lemma]\label{lem:transporter0}
For $\frq$, $\frb$, $\frb_0$ as above, we have
\begin{gather} 
\nor_{\frq}(\frb,\frb_0) = \frq\cap[\frg,\frb]^\perp,
\label{eq:transporter} \\
\codim_{\frq} \nor_{\frq}(\frb,\frb_0)\ \leq \ \codim_{\frq} \frq \cap \frb^{\perp}\ \leq\  \dim \frb - \dim \frb_{0}  \ \leq\ \ell.
\label{eq:transporter_codim}
\end{gather}
\end{lem}
\begin{proof}
Let  $Z\in\frb$ and $X \in \frq$ and  $Y\in\frg$. Since $Z \in \gs$,  we have $\langle [Y,Z], X\rangle
=-\langle Y,[ X,Z]\rangle$.
This  shows  the equivalence of
$X\perp [\frg,\frb]$ and $X\in\nor_{\frq}(\frb,\frb_0)$.
Hence, the equation \eqref{eq:transporter} holds.

As $[\frg,\frb]\subseteq \frb$ and thus
$[\frg,\frb]^\perp\supseteq \frb^\perp$,
we clearly have
$\codim_{\frq} \frq \cap [\frg ,\frb]^\perp  \leq \codim_{\frq} \frq \cap \frb^{\perp}$.
Now $\codim_\frg\frb^\perp = \dim \frb - \dim \frb \cap \frg^{\perp}$. 
Since $\frb$ is totally isotropic, 
this means $\codim_\frg\frb^\perp  \leq \ell$.
Since  $\codim_{\frq} \frq \cap \frb^{\perp} \leq \codim_\frg\frb^\perp$, 
the inequalities \eqref{eq:transporter_codim} follow. 
\end{proof}

\begin{remark} Equality holds in \eqref{eq:transporter_codim} 
if and only if $\dim\frq + \frb^{\perp} = \dim \frg$.
\end{remark}

The following relations between transporters are satisfied: 
\begin{lem} \label{lem:transrel} \hspace{1ex}
\begin{enumerate}
 \item $[\gs, \nor_{\frg}(\frb,\frb_0)] \subseteq \frb^{\perp} \cap \gs \subseteq \nor_{\gs}(\frb,\frb_0)$
 \item $ \nor_{\gs}(\frb,\frb_0)$ is an ideal of $\frg$.
 \item $\frb^{\perp} \subseteq \nor_{\frg}(\frb,\frb_0)$.
\end{enumerate}
\end{lem}
\begin{proof}  
Let $Y \in \gs$  and  $X \in  \nor_{\frg}(\frb,\frb_0)$, 
$Z \in  \frb$. Since $ [ X, Z] \in \frb_{0}$, we get 
$\langle  [Y,X] , Z] \rangle = \langle  X, [Z, Y]  \rangle =  \langle   [X, Z], Y  \rangle
= 0 $.  This shows $[Y,X] \perp \frb$.  By \eqref{eq:transporter},  $\frb^{\perp} \cap \gs \subseteq  \nor_{\gs}(\frb,\frb_0)$. This shows (1). 

Let $Y \in \frg$  and  $X \in  \nor_{\gs}(\frb,\frb_0)$, 
$Z \in  [\frg,\frb]$. We get $\langle  [Y,X] , Z  \rangle = \langle  Y, [Z, X]  \rangle =  \langle   [Y, Z], X  \rangle$. Since  $ [Y, Z] \in  [\frg,\frb]$,  \eqref{eq:transporter} shows that  $\langle  [Y,X] , Z \rangle = 0$. Thus,  $[Y,X] \in \nor_{\gs}(\frb,\frb_0)$. This shows (2). 

Finally, $\frb^{\perp} \subseteq [\frg, \frb]^{\perp}$, which, in light of  \eqref{eq:transporter},  shows (3).  
\end{proof}

\subsubsection{Totally isotropic ideals} 

\begin{prop} \label{prop:transporter_ideals}
Let $\fri$ be an ideal of $\frk$ contained in $\nor_{\frk}(\frb,\frb_0)$. 
Then  $\fri +  \nor_{\gs}(\frb,\frb_0)$ is an ideal of $\frg$ contained in  $\nor_{\frg}(\frb,\frb_0)$. 
In particular,  $[\fri + \nor_{\gs}(\frb,\frb_0), \frb]$ is an ideal of $\frg$ contained in $\frg^{\perp}$. 
\end{prop}
\begin{proof}  
Clearly,  $\frj = \fri +  \nor_{\gs}(\frb,\frb_0)$ is contained in  $\nor_{\frg}(\frb,\frb_0)$.  
Recall that $\frg = \frk + \gs$. Since $\fri$ is an ideal in $\frk$, $[ \frk , \frj] \subseteq \fri + [\frk, \nor_{\gs}(\frb,\frb_0)]$.
Using (2) of Lemma \ref{lem:transrel}, we conclude   $[ \frk , \frj] \subseteq \frj$.  By (1) of  Lemma \ref{lem:transrel}, $[ \gs , \frj] \subseteq  \nor_{\gs}(\frb,\frb_0) \subseteq \frj$. 
\end{proof} 

\begin{cor} \label{cor:isotropideal_ab}
Assume that  $\frg^\perp$ does not contain any non-trivial
ideal of $\frg$. Then  every  totally isotropic ideal $\frb$ of $\frg$ 
contained in $ \gs$ is abelian. 
\end{cor} 
\begin{proof}
 By (3) of Lemma \ref{lem:transrel}, $\frb \subseteq \nor_{\gs}(\frb, \frb_{0})$. Therefore, 
 $[\frb, \frb]$ is an ideal of $\frg$ and contained in $\frg^{\perp}$. Hence, $[\frb, \frb] = \zsp$.   
\end{proof}

The case of a large transporter in $\frk$ has particularly strong consequences: 
\begin{prop} \label{prop:transporter_k} 
Assume that $\nor_\frk(\frb,\frb_0) = \frk$. 
Then: 
\begin{enumerate} 
\item $\frb \cap \frg^{\perp}$ is an ideal in $\frg$. 
\end{enumerate} 
If  furthermore $\frg^\perp$ does not 
contain any non-trivial ideal of $\frg$, then: 
\begin{enumerate}  \setcounter{enumi}{1}
\item $\frb \cap \frg^{\perp} = \zsp$,  $\dim \frb \leq \ell$,  $[\frk, \frb] =  \zsp$. 
\end{enumerate}
\end{prop}
\begin{proof} Since, by Theorem \ref{mthm:invariance}, $\met$ is invariant by $\gs$, 
$[\gs,  \frg^{\perp}] \subseteq  \frg^{\perp}$. Hence, $[\gs, \frb \cap  \frg^{\perp}] \subseteq  \frb \cap  \frg^{\perp}$. Since $\nor_\frk(\frb,\frb_0) = \frk$ this implies that $\frb_{0} =  \frb \cap  \frg^{\perp}$ is an ideal in $\frg$. 
\end{proof} 

\subsubsection{Metric radical of $\gs$}
In  the following consider the special case $$ \frb=\gs^\perp\cap\gs \,  . $$ Thus $\frb$ is totally isotropic and it is the metric radical of $\gs$ (with respect to the induced metric  $\met_{\gs}$).
By Theorem \ref{mthm:invariance}, $\met_{\gs}$ is invariant by $\frg$.
Therefore, $\frb$ is an  ideal in $\frg$.  Moreover,  
$$ \frb_0  =  \frg^\perp \cap  \gs $$  is an ideal in $\gs$. 
\begin{lem}\label{lem:SRGperp} 
 $[\gs, \gs^{\perp}]  \subseteq \frb_{0}$. In particular, $[\gs, \frb] \subseteq \frb_{0} \subseteq \frg^\perp$.
\end{lem}
\begin{proof}
Let $Y \in \frg$, $X \in \gs$, $B \in \gs^{\perp}$. Since $\gs$ is an
ideal, $[Y,X] \in \gs$.  
Since $\met$ is $\gs$-invariant, we obtain 
$\langle Y, [X, B] \rangle = \langle [Y, X], B  \rangle
= 0$. This shows $\frg \perp  [\gs, \gs^{\perp}]$. 
\end{proof}

Since $[ \gs, \frb]$ is an ideal in $\frg$,  we deduce: 

\begin{cor}\label{cor:GSGperp}
If\ $\frg^\perp$ does not contain a non-trivial ideal of\ $\frg$, 
then 
\begin{enumerate}
\item[]  $ \gs^\perp\cap \gs  \subseteq \centre{\gs}$. In particular, $\frg^\perp\cap \gs  \subseteq \centre{\gs}$.
\end{enumerate}
\end{cor}

 
The following strengthens  Proposition \ref{prop:transporter_k} for  
$\frb = \gs^{\perp} \cap \gs$ and $\frb_0 =  \frg^\perp \cap \gs $:

\begin{prop} \label{prop:transporter_k_gs}
 Assume that $\nor_\frk(\frb,\frb_0) = \frk$. 
Then: 
\begin{enumerate}
\item  
$[\frg,  \gs^{\perp} \cap \gs ]  \subseteq   \frg^\perp \cap \gs $.  In particular, $ \frg^\perp \cap \gs $ is 
an ideal in $\frg$. 
\end{enumerate} 
If furthermore $\frg^{\perp}$ contains no non-trivial ideal of\ $\frg$,
then:
\begin{enumerate} 
\setcounter{enumi}{1}
\item  $\frg^{\perp} \cap \gs = \zsp$ and $[\frg,  \gs^{\perp} \cap \gs ] = \zsp$. 
\item $[ \gs^\perp, \gs] =  [\frg^\perp, \gs] = \zsp$. 
\end{enumerate}
\end{prop}

\begin{proof}  By assumption, $[\frk, \frb] \subseteq\frb_0$. 
By  Lemma \ref{lem:SRGperp},
$[\gs, \frb] \subseteq\frb_{0} \subseteq\frg^\perp$, so that 
 $[\frg, \frb] \subseteq\frb_{0}$. In particular, $\frb_{0}$ is an ideal in $\frg$.
Thus (1), (2) follow, and also (3), since $[ \gs^\perp, \gs] \subseteq\frb_{0}$, by  Lemma \ref{lem:SRGperp}. 
\end{proof} 

\begin{remark}
It is not difficult to see (compare Lemma \ref{lem:rep_kernel} below) that the centralizer of $\gs$ in $\frg$
is $\frz_{\frg}(\gs) =  \frz_{\frk}(\gs) \times \frz(\gs)$. 
\end{remark} 

\subsection{Metric radical of $\frg$} \label{sec:metric_radical}
\begin{lem} \label{lem:rep_kernel}
Let $W$ be a $\frg$-module. Suppose that  $\frc = \{ Z \in \frg \mid Z \cdot W = \zsp \}$, the centralizer of $W$, is contained in $\frk + \frr$. Then
$\frc  = (\frc  \cap \frk) + (\frc  \cap \frr)$.
\end{lem}
\begin{proof}  Let $\frg = \frf \ltimes \frr$ (where $\frf \supseteq \frk$ is a semisimple subalgebra, and $\frr$ the maximal solvable ideal of $\frg$) be a Levi decomposition of $\frg$. 
Assume first that $W$ is an irreducible $\frg$-module. Then the action of $\frr$ on $W$ is reductive and commutes with $\frf$. Since the image of $\frf$ in $\gl(W)$ has trivial center, the claim of the lemma follows in this case.
For the general case, consider a H\"older sequence of submodules $W \supseteq W_{1} \supseteq \ldots \supseteq W_{k} = \zsp$ such that 
the $\frg$-module $W_{i}/W_{i+1}$ is irreducible. The above implies that, for any $Z = K + X  \in \frc $, where $K \in \frf$ and $X \in \frr$, $K$ (and $X$) act trivially on  $W_{i}/W_{i+1}$. Since $K \in \frk$ is semisimple on $W$, this implies that $K$ acts trivially on $W$. That is, $K \in  \frc  \cap \frk$ and therefore also $X \in \frc  \cap \frr$. 
\end{proof}
 
 \begin{prop} \label{prop:gsperp} If\ $\frg^\perp$ does not contain a non-trivial ideal of\ $\frg$, 
then 
\begin{enumerate}
\item $ \gs^\perp \subseteq \frn_{\frk}(\gs, \centre{\gs} \cap \gs^{\perp}) + \frn_{\frn}(\gs, \centre{\gs}  \cap \gs^{\perp})$.  
\item $ \frg^\perp \subseteq \frn_{\frk}(\gs, \centre{\gs} \cap \frg^{\perp}) + \frn_{\frn}(\gs,  \centre{\gs} \cap \frg^{\perp})$.  
\end{enumerate}
\end{prop}
\begin{proof} By Lemma \ref{lem:SRGperp}, $[ \gs^\perp, \gs ] \subseteq \frb_{0} = \gs \cap \frg^{\perp} \subseteq \gs \cap \gs^{\perp} = \frb$. Since $\frb$ is an ideal of $\frg$, $W = \gs/\frb$ 
is a $\frg$-module. Now $\gs^\perp$ is contained in  $\frc  = \frn_{\frg}(\gs, \gs \cap \gs^{\perp})$, 
which is the centralizer of $W$.  In view of our assumption on ideals in $\frg^\perp$, observe that  $[\frc , \gs] \subseteq  \frb \subseteq \frz(\gs) \subseteq \frr$ by Corollary \ref{cor:GSGperp}.  Now $[\frc , \gs] \subseteq \frr$ implies that $\frc $ is contained in  
$\frk + \frr$.
Therefore, Lemma \ref{lem:rep_kernel} applies, showing  $ \gs^\perp \subseteq \frn_{\frk}(\gs, \gs \cap \gs^{\perp}) + \frn_{\frr}(\gs,  \gs \cap \gs^{\perp})$. 
 Since $[\frr, \centre{\gs}] = \zsp$, it follows that $\frn_{\frr}(\gs,  \centre{\gs}) \subseteq \frn_{\frn}(\gs,  \centre{\gs})$. Hence, (1) holds. 
 
 To prove (2), suppose $Z = K + X \in \frn(\gs, \gs \cap \frg^{\perp})$, 
where $K \in \frk$, $X \in \frr$.  
By (1), $K \in \frc  = \frn(\gs, \gs \cap \gs^{\perp})$. 
Since $K$ acts as a semisimple derivation on $\gs$, we can decompose 
$\gs = W_{1} +  (\gs \cap \gs^{\perp})$, where $[K, W_{1}] = \zsp$.  Now, for $w \in \gs$, write 
$w= w_{1} + v$, where $w_{1} \in W_{1}$, $v \in \gs \cap \frg^{\perp}$. 
Note that  $0 = \langle [Z, v] , Y \rangle =   \langle [K, v] , Y \rangle +  \langle [X, v] , Y \rangle $  for all $Y \in \frg$.
By Lemma \ref{lem:SRGperp},  $[X, v]  \in \frg^{\perp}$. This 
implies $[K, w] = [K, v] \in \frg^{\perp}$. It follows also that $[X, w] \in \frg^{\perp}$. 
\end{proof}

\begin{lem} \label{lem:trans_ideal}
Let $\frj = \frn_{\frn}(\gs, \centre \gs  \cap \frg^{\perp})$. Then $\frj$ is an ideal in $\frg$. 
\end{lem}
\begin{proof} Let $N \in \frj$, and $K, Y \in \frg$, $v \in \gs$. Then $\langle [[K, N], v], Y \rangle = \langle [K, N], [v, Y] \rangle =  \langle K, [N,  [v, Y]] \rangle = 0$, since $v' =  [v, Y] \in \gs$  and 
therefore, $[N, v'] \in \frg^{\perp}$. This shows $[ [K, N], \gs] \subseteq \frg^{\perp}$. 

Similarly, $ [[K, N], v] = [[K,v], N] + [[v,N], K]$, where $v' = [K,v] \in \gs$ and $[v', N] \in \centre \gs$, 
as well as $[v,N] \in \centre \gs$ and therefore,   $[[v,N], K] \in \centre \gs$. It follows $[ [K, N], \gs] \subseteq \centre \gs$. We conclude that $[K,N] \in \frj$. Hence, $\frj$ is an ideal of $\frg$.
\end{proof}

These considerations yield the following important property of
$\frg^\perp$: 

\begin{thm}  \label{mthm:gperp} 
Suppose that $\frg^\perp$ does not contain a non-trivial ideal of\ $\frg$. Then 
$$ \frg^\perp \subseteq \; \frn_{\frk}\left(\gs, \centre{\gs} \cap \frg^{\perp}\right) + \frz(\gs) . $$ 
\end{thm} 
\begin{proof} Consider the ideal $\frj$, as defined in Lemma \ref{lem:trans_ideal}. By (2) of Proposition \ref{prop:gsperp}, we have $\frg^\perp \subseteq \, \frn_{\frk} \left(\gs, \centre{\gs} \cap \frg^{\perp}\right)  + \frj$.  Since $\frj$ is an ideal in $\frg$, so is $[\frj, \gs]$. Since $[\frj, \gs] \subseteq \frg^{\perp}$,   the assumption on ideals in $\frg^{\perp}$ implies  that $[\frj, \gs] = \zsp$. It follows that $\frj$ is contained in $\centre \gs$. 
\end{proof}

%
\subsubsection[Invariance by $\frg^\perp$]{Invariance by ${\frg^\perp}$}

We shall be interested in nil-invariant bi\-linear forms $\met$ on
$\frg$ induced by pseudo-Riemannian metrics on homogeneous spaces.
In this case, $\met$ is invariant by the stabilizer subalgebra
$\frg^\perp$. We can then further sharpen the
statement of Corollary \ref{mcor:invariance_radical}.

\begin{prop} \label{prop:gperp_inv}
Let $\frg$ and $\met$ be as in Corollary \ref{mcor:invariance_radical}.
If in addition $\met$ is $\frg^{\perp}$-invariant,
then
\[
[\frg^{\perp}, \gs] = \zsp.
\]
\end{prop}


The proof is based on the following
immediate observations:

\begin{lem}\label{lem0}
Suppose $\met$ is $\frg^{\perp}$-invariant.
Then $ [ [\frk, \frg^{\perp}], \gs]  \subseteq \frg^{\perp} \cap \gs$. 
\end{lem} 

and

\begin{lem} \label{lemA}
Let $\frh$ be any Lie algebra and $V$ a module for $\frh$. Suppose that
the subalgebra $\frq$ of $\frh$ is generated by the subspace $\frm$  of $\frh$. 
Then $\frq \cdot V = \frm \cdot V$. 
\end{lem} 

Together with 

\begin{lem} \label{lemB}
Let $\frk$ be semisimple of compact type and $\frk_{0}$ a subalgebra of $\frk$. 
Then the subalgebra $\frq$ generated $\frm = \frk_{0} + [\frk, \frk_{0}]$ is an ideal of $\frk$. 
\end{lem} 
\begin{proof} 
Put $\frz= \frz_{\frk}(\frk_{0})$. 
Then $[\frz, \frm] \subseteq  \frm$ and  $[ [\frk, \frk_{0}], \frm] \subseteq  \frm +  [\frm, \frm]$. 
Since $\frk =  [\frk, \frk_{0}] + \frz$, this shows $[\frk, \frm] \subseteq \frq$.
Since $\frq$ is linearly spanned by the iterated commutators of elements of $\frm$,
$[\frk, \frq] \subseteq \frq$. 
\end{proof}

\begin{proof}[Proof of Proposition \ref{prop:gperp_inv}] 
Let $\frk_{0}$ be the image of $\frg^{\perp}$ under the projection  homomorphism $\frg \to \frk$. 
Note that by Corollary \ref{mcor:invariance_radical},
$[\frg^{\perp}, \gs] = [\frk_{0}, \gs]$.
Let $\frq \subseteq \frk$ be the  subalgebra generated by
$\frm = \frk_{0} + [\frk, \frk_{0}]$ and consider $V = \gs$ as a module
for $\frq$.
Since $\frq$ is an ideal of  $\frk$, $[\frq, V]$ is a submodule for $\frk$, that is, $[\frk, [\frq, V]] \subseteq  [\frq, V]$.
By Lemmas \ref{lem0}, \ref{lemA} and Corollary \ref{mcor:invariance_radical}
we have
$[\frq, V] = [\frm, V] \subseteq \frg^{\perp} \cap \zen(\gs)$.
Hence, $\frj =  [\frm, V] \subseteq \frg^{\perp}$ is an ideal in $\frg$, with $\frj \supseteq [\frg^{\perp}, \gs] = [\frk_{0}, \gs]$.
Since $\frg^\perp$ contains no non-trivial ideals of $\frg$ by
assumption, we conclude that $\frj=\zsp$.
\end{proof}


\subsection{Transporter in $\frk$ and low relative index} 

\begin{lem}\label{lem:subalg_k}
Let $\frk$ be a semisimple Lie algebra of compact type and
$\frl$ a sub\-algebra of $\frk$.
Then either $\frl = \frk$ or $m = \codim_{\frk} \frl >1$.
Assume further that $\frl$ does not contain any non-trivial
ideal of $\frk$.
Then, up to conjugation by an automorphism of $\frk$:  
\begin{enumerate}
\item
if $m=2$, then
$\frk = \so_3$ and  $\frl= \so_2$,
\item
if $m=3$, one of the following holds:
\begin{enumerate}
\item[\rm (a)]
$\frk=\so_3$ and $\frl=\zsp$,
\item[\rm (b)]
$\frk=\so_3\times\so_3$, 
$\frl$ is the image of a diagonal embedding $\so_3\to\so_3\times\so_3$.
\end{enumerate}
\end{enumerate}
\end{lem}
\begin{proof}

As an $\ad_\frk(\frl)$-module, $\frk=\frl\oplus\frw$
for a  submodule $\frw$. For this,  note that any Lie subalgebra of $\frk$ acts reductively,
since $\frk$ is of compact type. 

Suppose $\codim_\frk\frl=1$, that is, $\frw$ is one-dimensional. 
Then $[\frw,\frw]=\zsp$  and it follows that $\frw$ is also an ideal of $\frk$. 
A one-dimensional ideal cannot exist, since $\frk$ is semisimple. 
It follows that $\codim_\frk\frl>1$.

Since $\frk=\frl\oplus\frw$,  the kernel of the adjoint action of $\frl$ on $\frw$
is an ideal in $\frk$.  Assume further that $\frl$ contains no non-trivial
ideals of $\frk$. Then $\frl$ acts faithfully on $\frw$. 

For $m=2$, this means $\frl =\so_2$ and
$\dim\frk=m+\dim\frl=3$. Hence,  $\frk=\so_3$.

For $m=3$, $\frl$ embeds into $\so_3$. If $\frl=\zsp$, we have $\dim\frk=3$ and 
thus $\frk = \so_3$. Otherwise,  either $\frl  =  \so_{2}$ or $\frl = \so_3$.
In the first case, $\dim \frk = 4$. Since there is no four-dimensional simple Lie
algebra, this is not possible.   
In the latter case,  $\dim\frk=6$. This leaves $\frk=\so_3\times\so_3$ (being isomorphic to $\so_4$)
as the only possibility.  Since $\frl$ is not an ideal of $\frk$, 
$\frl$ projects injectively onto both factors of $\frk$.  
It follows that, up to automorphism of $\frk$,  $\frl$ is the image of an 
embedding $\so_3\to\so_3\times\so_3$, $X\mapsto(X,X)$. 
\end{proof}


\subsubsection{Totally isotropic ideals and low relative index}  \label{sec:transporter_l2}

Let $\frb$ be any totally isotropic ideal of $\frg$ contained in $\gs$ and
put $\frb_{0} = \frb \cap \frg^{\perp}$. 

%

\begin{prop}\label{prop:transk}
If $\ell \leq 2$, then $\nor_{\frk}(\frb,\frb_0) = \frk$. 
\end{prop} 
\begin{proof}
Put $\frl=\nor_\frk(\frb,\frb_0)$ and  $m=\codim_{\frk}\frl$. 
By Lemma \ref{lem:transporter0}, $\frl = \frk \cap [\frg, \frb]^{\perp}$ 
and  $m\leq\ell$. 

Assume now that $m \geq 1$. According to Lemma \ref{lem:subalg_k},  the case $m=1$ never occurs. 
Hence, in this case, we have $m =  2$. 

Let $\fri\subseteq \frl$ be the maximal ideal of $\frk$
contained in $\frl$. Using Proposition \ref{prop:transporter_ideals}, we 
see that there exists an ideal $\frb_{1}$ of $\frg$, such that $[\fri, \frb] \subseteq \frb_{1}
 \subseteq  [\frg,\frb] \cap \frg^{\perp}$.
Since $[\frg,\frb]$ and $\frb_{1}$ are ideals, $U  =  [\frg,\frb]/\frb_{1}$ is a module
for $\frk$. In fact, since  $[\fri, \frb] \subseteq \frb_{1}$, $U$ is a module for $\frk/\fri$. 
Also, since $\frb_{1} \subseteq \frg^{\perp}$, $\met$ restricted to $\frk  \times  [\frg,\frb]$ 
induces a skew pairing on $\frk \times  U$, such that  $U^{\perp} = \frl$. Since  
we have $\fri \perp [\frg,\frb]$, this shows that 
%
$\met$ restricted to  $\frk  \times  [\frg,\frb]$ descends to a skew pairing  
\begin{equation} \label{eq:skewp}
\met:(\frk/\fri) \times U \to \RR \; , \text{ where $U^{\perp} = \frl/\fri$.}
\end{equation} 
If  $m=2$, then by Lemma \ref{lem:subalg_k},
$\frk/\fri = \so_3$ and  $\frl/\fri = \so_2$. 
By Corollary \ref{cor:skewkernel}, either the skew pairing $\met$ in \eqref{eq:skewp} is zero 
(that is,  $U^{\perp} = \frk/\fri$) or $\frk \cap [\frg ,\frb]^\perp = \fri$. 
In the first case, $\frl = \frk \cap [\frg ,\frb]^\perp = \frk$. 
In the second case,  $\frl = \fri$, a contradiction to $\frl/\fri= \so_2$. 
Therefore, $m=0$. 
\end{proof}

Combining with Proposition \ref{prop:transporter_k} (1) we arrive at: 

\begin{cor}  \label{cor:isotrop_ideals_l2_0}
If $\ell \leq 2$ then, for any  totally isotropic ideal $\frb$ of $\frg$ contained in $\gs$, $\frg^{\perp} \cap \frb$ is an ideal in $\frg$. In particular, $\frg^{\perp} \cap \gs$ is an ideal in $\frg$.
\end{cor}

The following now summarizes our results on totally 
isotropic ideals in case $\ell \leq 2$:

\begin{cor} \label{cor:isotrop_ideals_l2} 
 Assume that $\frg^\perp$ does not contain any non-trivial ideal of
$\frg$ and  that $\ell \leq 2$. 
Then, for any  totally isotropic ideal  $\frb$  of $\frg$ contained in
$\gs$, 
\begin{enumerate}  
\item $\frb \cap \frg^{\perp} = \zsp$,  $\dim \frb \leq \ell$,  $[\frk, \frb] = \zsp$. 
\end{enumerate} 
Furthermore, the following hold:  
\begin{enumerate} \setcounter{enumi}{1}
\item $\frg^{\perp} \cap \gs = \zsp$. 
\item $[\gs^{\perp} , \gs] = \zsp$. 
\item $[\frg, \gs \cap \gs^{\perp}] = \zsp$. 

\end{enumerate}
\end{cor}
\begin{proof}
Since $\ell \leq 2$,  according to Proposition \ref{prop:transk} $\nor_\frk(\frb,\frb_0) = \frk$. 
Thus (1) holds due to part (2) of Proposition \ref{prop:transporter_k}.
%

Now (2), (3), (4) are consequence of Proposition \ref{prop:transporter_k_gs}.
\end{proof}

Combining with Theorem \ref{mthm:gperp}, we also obtain: 
\begin{cor} \label{cor:gperp_l2} 
Assume that $\frg^\perp$ does not contain any non-trivial ideal of
$\frg$ and  that $\ell \leq 2$. Then $\frg^\perp$ is contained in $\frz(\gs) \times \frk$ 
and  $\frg^\perp \cap \gs  = \zsp$. 
\end{cor}

\begin{remark}
As Example \ref{ex:so3so4} shows, these conclusions do not necessarily hold if $\ell\geq 3$. 
\end{remark} 
\subsection{Metric radicals of the characteristic  ideals} 

This section serves to clari\-fy the relations between the metric
radicals of $\gs$, $\frr$ and $\frn$,
where $\frn$ denotes the nilradical of $\frr$. 


\begin{lem}\ \label{lem:jperp_2}
\begin{enumerate}
\item
$[\frr, [\frg,\frr]^{\perp}] \perp \frg$. 
\item
$[\frr,  \frn^\perp] \perp \frg$
and $[\frg, \frn^\perp \cap \gs]  \perp \frr$.
\item
$[\gs, (\frs + \frn)^\perp] \perp \frg$
and $[\frg,  (\frs + \frn)^\perp\cap\gs] \perp \gs$.
\item
$[\gs, \frn^\perp] \perp (\frk + \frr)$
and $[\frk + \frr,  \frn^\perp \cap \gs] \perp \gs$.
\end{enumerate}
\end{lem}

\noindent 
The lemma is clearly implied by: 

\begin{remark} \label{rem:abc} 
Let $\fra \subseteq\inv(\frg,\met)$, $\frb, \frc \subseteq\frg$ be
subspaces such that $[\fra, \frc] \subseteq\frb$. Then 
$[\fra, \frb^{\perp}] \perp \frc$.  Furthermore, this implies
$\fra \perp  [\frb^{\perp} \cap \inv(\frg,\met), \frc]$. 
\end{remark} 

\begin{lem}\label{lem:jperp_1}
Let $\frj \subseteq\gs$ be an ideal in $\frg$. Then the  following hold:
\begin{enumerate}
\item
$\frj^\perp\cap\gs$ and $\frj^\perp\cap\frj$ are ideals of $\frg$.
\item
$[\frj, \frj^\perp] \subseteq \frg^{\perp}$.
\end{enumerate}
\end{lem} 
\begin{proof}
Since $\met$ restricted to $\gs$ is $\frg$-invariant by Theorem
\ref{mthm:invariance},
$\frj^\perp\cap\gs$ is an ideal in $\frg$.
It follows that  $\frj^\perp\cap\frj$ is an ideal.
Hence (1) holds. Now (2) follows using the above remark 
with $\fra = \frj$, $\frc = \frg$ and $\frb = \frj$. 
\end{proof}

\subsubsection{Radicals in effective metric Lie algebras} 

For all following results,  we shall also require that the metric Lie algebra $(\frg, \met)$ is 
effective.  That is, we assume for now that \emph{$\frg^\perp$ does not contain any non-trivial
ideal of $\frg$}. 

\begin{lem}  \label{lem:jperp_effective}
Let $\fri, \frj  \subseteq\gs$ be ideals in $\frg$. Then: 
\begin{enumerate}
\item
$[\frj, \frj^\perp \cap \gs]  = \zsp$ and
$\frj^\perp \cap \frj = \frj^\perp \cap \frz(\frj)$.
\item
If 
$\zen(\fri) \subseteq \frj \subseteq \fri$ then
$
\fri^\perp\cap \fri\ \subseteq\ \frj^\perp\cap \frj 
$. 
\end{enumerate}
\end{lem} 
\begin{proof}
By Lemma \ref{lem:jperp_1} (1), $[\frj, \frj^\perp \cap \gs]$ is an ideal of $\frg$
and contained in $\frg^\perp$.
Since $\frg^\perp$ does not contain any 
non-trivial ideal of $\frg$, 
 $[\frj, \frj^\perp \cap \gs] = \zsp$. 
Hence, (1) holds.
Under the assumption of  (2), this means
$\fri^\perp\cap \fri \subseteq\zen(\fri) \subseteq \frj$.
Since also $\fri^\perp \subseteq \frj^\perp$, (2) follows.
\end{proof}

%

The next result somewhat strengthens  Corollary \ref{cor:GSGperp}. 
\begin{prop} \label{prop:jperp_4}
The following hold: 
\begin{enumerate}
\item
$[\frr,   [\frg,\frr]^\perp \cap \gs] = \zsp$.
\item
$[\frr,  \frn^\perp \cap \gs] = \zsp$.
In particular,  $ \frn^\perp  \cap  \frr \subseteq \frz(\frr)$.
\item
$[\gs, (\frs + \frn)^\perp\cap\gs] = \zsp$.
In particular, $(\frs + \frn)^\perp \cap\gs\subseteq\frz(\gs)$.
\end{enumerate}
\end{prop}
\begin{proof}
By Lemma \ref{lem:jperp_1} (1),
$\frj^\perp\cap\gs$ is an ideal of $\frg$ for 
any ideal $\frj$ of $\frg$ contained in $\gs$.
Then for any ideal $\fri$ of $\frg$,
$[\fri,\frj^\perp \cap\gs]$ is also an ideal in $\frg$.  
Therefore,  if
$[\fri,\frj^\perp \cap\gs]\subseteq \frg^\perp$, then
$[\fri,\frj^\perp \cap\gs]=\zsp$.
In the view of Lemma \ref{lem:jperp_2}, (1), (2), (3) follow.  
\end{proof}

We can deduce from (2) of  Proposition \ref{prop:jperp_4}
the equalities 
\begin{gather} 
\frn^\perp \cap \frr
=\frn^\perp \cap \frn
=\frn^\perp \cap \zen(\frn)
=\frn^\perp \cap \zen(\frr) ,\\
\frr^\perp \cap \frr
=\frr^\perp \cap \frn
=\frr^\perp \cap \zen(\frn)
=\frr^\perp \cap \zen(\frr).
\end{gather}
Also (3) of  Proposition \ref{prop:jperp_4} shows that 
\begin{equation} 
\label{eq:jperptower2}
\gs^\perp\cap\gs \subseteq \zen(\gs)\subseteq\zen(\frr)  , 
\end{equation}
Moreover,  using nil-invariance of $\met$
and Corollary \ref{cor:ksperp} (1), the above  yield 
\begin{equation}
\label{eq:jperptower3}
[\frg,\frn^\perp \cap \frn] \subseteq  \frr^\perp \cap \frr,
\
[\frs,\frn^\perp\cap\frn]\subseteq\frr^\perp\cap\frr\cap\frk^\perp
\ \text{ and }\ 
[\frk + \frr, \frn^\perp \cap \frn] \subseteq  \gs^\perp  \cap  \frr . 
\end{equation} 
Thus there is a  tower of totally isotropic ideals
of $\frg$ contained in $\frz(\frr)$:  
\begin{equation}\label{eq:jperptower}
\gs^\perp\cap\gs\ \subseteq\
\frr^\perp  \cap  \frr  \ \subseteq\
\frn^\perp  \cap  \frn  .  
\end{equation} 

%
%


%
\subsection{Actions of semisimple subalgebras on the solvable radical} \label{sec:fsubmodules} 
 Let $\frq$ be a subalgebra of $\frg$. We call the subspace $W \subseteq \frg$ a submodule for $\frq$ 
 if $[\frq, W] \subseteq W$. In the following we let $\frf \subseteq \frg$ denote a semisimple subalgebra of $\frg$. 
 As usual, we decompose $\frf = \frk \times \frs$, where $\frk$ is an ideal of compact type and 
 $\frs$ has no factor of compact type.  

\begin{lem}\label{lem:fsubmodules_perp}
Let 
$W \subseteq \frg$ be a submodule for $\frf$,  with $\dim [\frf, W]  \leq 2$. Then: 
\[
\frf \perp  [\frf, W], \ [\frk, W]  = \zsp, \ \frs \perp W.
\]
%
%
\end{lem}
\begin{proof}
Assume first that $W$ is not a trivial module. Thus $\dim [\frf, W ] = 2$ and 
$\frf = \frf_0 \times \sl_2(\RR)$, where $\frf_0$ is the kernel of
the representation of $\frf$ on $W$.
 As $W  \subseteq \frg^{\frf_0}$, 
 Lemma \ref{lem:gk_orthog} states that $[W, \gs]  \perp \frf_0$. Clearly,
\[
[W, \frf]  = [W,  \sl_2(\RR)] \ \subseteq \  [W, \gs]  .
\]
Therefore, $ \frf_0 \perp [W, \frf]  $ and $ [W, \frf]  \subseteq \gs$.    

Since $\gs \subseteq \inv(\frg, \met)$ (part (2) of Theorem \ref{mthm:invariance}), $\met$   induces a skew pairing $\sl_2(\RR) \times  [W, \frf]   \to \RR$ for the module $[W, \frf] $. 
Since $ [W, \frf] $ is of dimension two and non-trivial, 
Proposition \ref{prop:sl2_skew_module} shows that 
$\sl_2(\RR)\perp  [W, \frf] $. This now implies $\frf \perp  [\frf, W]$ and $[\frk, W]  = \zsp$, as
$\frk \subseteq \frf_{0}$. 
Since $\frs = [\frs,\frs] \subseteq \inv(\frg,\met)$, 
$[\frs,W] \perp \frs$  implies that $W \perp \frs$. 
\end{proof}

\begin{lem}\label{lem:com_q_submodules} Let $\frq$ be a subalgebra of $\frg$, and 
let\/ $W \subseteq \frr$ be a submodule for $\frq$. 
Then the  following hold  for\/ $\frl = \frq \cap [W,W]^{\perp}$:  
\begin{enumerate}
\item $\frl$ is a subalgebra, and $[\frl, W]  \subseteq W^{\perp}$. 
\end{enumerate} 
Assume further that $\frq$ acts reductively on $W$. Then: 
\begin{enumerate}
\item[(2)] $\frl = \frq \cap [W_{1},W_{1}]^{\perp}$, where $W_{1} = [\frq, W]$.  
\item[(3)] $\frl$ is an ideal in $\frq$. 
\end{enumerate} 
If $\frq =\frf$ is semisimple and $\dim  [W_{1},W_{1}] \leq 2$ then: 
\begin{enumerate}
\item[(4)] $\frl = \frf$ and $\frf \perp [W,W] $. 
\item[(5)] $[\frf, W]$ is totally isotropic. 
\end{enumerate}  
\end{lem}
\begin{proof}
Observe that for any $u,v \in W$, $K \in \frg$, $\langle K, [u, v ] \rangle = \langle [ K, u],  v  \rangle$. 
In particular, $K \perp [W,W]$ is equivalent to $[K,W] \perp W$.  
To finish the proof of (1), assume that $K_{1}, K_{2} \perp 
 [W,W]$, where $K_1,K_2\in\frq$.
Then  also $ \langle [ [K_{1}, K_{2}] , u],  v  \rangle = \langle [ [K_{1}, u] , K_2],  v  \rangle + \langle [ [u,K_{2} ] , K_1],  v  \rangle =0$.  
 Hence, $ [K_{1}, K_{2}] \perp [W, W]$. This shows that $\frl$ is a subalgebra. 
 

Next we show (2). Since $\frq$ acts reductively on $W$, $W= W_0\oplus W_{1}$, with $W_1 = [\frq,W]$
and $[\frq, W_0]=\zsp$. For any $u, v \in W$, decompose $u = u_{0} + u_{1}$,
$v = v_{0} + v_{1}$, where $u_{i}, v_{i} \in W_{i}$. Then compute $\langle K, [u, v ] \rangle = \langle K, [u_{1}, v_{1} ] \rangle$.

Finally, if $\frq$ acts reductively, there is a decomposition into submodules $W = (W \cap W^{\perp}) \oplus W'$. 
Correspondingly, $K \in \frl$ if and only if $[K, W'] = \zsp$.  This shows that $\frl$ is an ideal in $\frq$. Hence, (3) holds.

If $\frf$ is semisimple, then $\frf$ acts reductively on $W$. By  part (3), $\frl = \frf \cap  [W,W]^{\perp}$
 is an ideal of $\frf$. Since  $\dim [W,W] \leq 2$ it is an ideal  of codimension at most two. Since $\frf$ 
 is semisimple this implies $\frl = \frf$. Hence, (4) holds.
Now (4) together with (1) implies that $[\frf, W] \subseteq W^{\perp}$ is totally isotropic. 
\end{proof}

For any subspace $W$ of $\frg$, recall that  $\mu(W)$ denotes
the index of $W$.  

\begin{lem} \label{lem:combounded_f_submodules}
Let 
$W \subseteq \frr$  be a submodule for $\frf$, such that $\dim [W,W] \leq 2$. Then: 
\begin{enumerate} 
\item $[\frf, W] \subseteq W^{\perp}$ is totally isotropic. 
\item If $\mu(\gs) \leq 2$ then $[\frf, W] = \zsp$.
\end{enumerate}
\end{lem}
\begin{proof}
 
 
By Lemma \ref{lem:com_q_submodules} part (5), 
$[\frf, W] \subseteq W^{\perp}$ is totally isotropic. 
In particular, assuming $\mu(\gs) \leq 2$, 
$\dim [\frf, W] \leq 2$.  
Then Lemma \ref{lem:fsubmodules_perp}  implies $[\frk, W] = \zsp$, $\frs \perp W$. 
Assuming $ [\frf, W]  \neq \zsp$, $\dim [\frf, W] = 2$ and $\frs$ contains $\sl_{2}(\RR)$, so that
$\mu(\frs) \geq \mu( \sl_{2}(\RR)) \geq 1$.  We get 
$2= \mu( [\frf, W] ) \leq \mu(\gs) - 1 \leq 1$. Thus  (2) follows. 
\end{proof}
 
We are ready to give the main result of this subsection.  
 
\begin{prop} \label{prop:commuting_ideals}
If $\mu(\gs) \leq 2$ then $[\frk\times\frs,\frr]=\zsp$. 
\end{prop}
\begin{proof}  We have  $\mu(\frr) \leq \mu(\gs) \leq 2$. Thus Proposition 
\ref{prop:index2algebras_q} implies that there exists an ideal $\frq$ of $\frg$ 
with $\dim [\frq,\frq] \leq 2$, and the codimension of $\frq$ in $\frr$ is at most two. 
Since $\mu(\gs) \leq 2$, $[\frk\times\frs,\frr]=\zsp$, by Lemma \ref{lem:combounded_f_submodules}.
This also implies $\frs \perp \frr$ (compare Lemma \ref{lem:fsubmodules_perp}). 
\end{proof}
 
As a consequence we further get: 

\begin{lem}\label{lem:sandr}
Suppose $\mu(\gs) \leq 2$. Then the following hold: 
\begin{enumerate}
\item
$\frs$ is non-degenerate.
\item
$\frs\perp(\frk+\frr)$ and $\frk\perp[\frr,\frr]$.
\item
$\mu(\frr) + \mu(\frs) \leq \mu(\gs)$.
\end{enumerate}
\end{lem}
\begin{proof} 
Note that  $\dim \frs\cap\frs^\perp \leq \mu(\gs) \leq 2$. 
Since $\met_\frs$ is invariant,  $\frs\cap\frs^\perp$
is an ideal in $\frs$.
We conclude that $\frs\cap\frs^\perp=\zsp$.  This shows (1).

Since $\met$ is invariant by $\frr$ and $\frs$, $[\frk\times\frs,\frr]=\zsp$ implies $\frk \perp [\frg,\gs] = \frs + [\frr, \frr]$ and
$\frs \perp \frr$. 
Hence,  (2) and (3) hold. 
%
\end{proof}

\section{Lie algebras with nil-invariant
scalar products of small index} \label{sec:classificationLie_sum} 

Partially summarizing the results from Proposition \ref{prop:commuting_ideals} 
and Corollary \ref{cor:gperp_l2} 
we obtain a first structure theorem for metric Lie algebras of
relative index $\ell\leq2$.

{\renewcommand{\themthm}{\ref{mthm:smallindex}}
\begin{mthm}
Let $\frg$ be a real finite-dimensional Lie algebra with
nil-invariant symmetric bilinear form $\met$ of relative index
$\ell\leq 2$,
and assume that $\frg^\perp$ does not contain a non-trivial ideal of $\frg$.
Then:
\begin{enumerate}
\item
The Levi decomposition \eqref{eq:levi} of\, $\frg$ is a direct sum of ideals:   $\frg=\frk \times \frs \times \frr$.
\item 
$\frg^\perp$ is contained in $\frk \times \frz(\frr)$ 
and  $\frg^\perp \cap \frr  = \zsp$. 
\item
$\frs\perp(\frk\times\frr)$ and $\frk\perp[\frr,\frr]$.
\end{enumerate}
\end{mthm}
}

We will now study the cases $\ell=0$, $\ell=1$ and
$\ell=2$ individually.

\subsection{Semidefinite nil-invariant products}
 Let $\met$ be a nil-invariant symmetric bilinear form on $\frg$.

\begin{prop}\label{prop:riemann}
If $\met$ is semidefinite (the case $\ell = 0$), then 
\begin{enumerate}
\item
$[\frg, \frs + \frr  ]\subseteq \frg^\perp$.
\end{enumerate}
Moreover,  if $\frg^\perp$ does not contain any non-trivial
ideal of $\frg$, then:
\begin{enumerate}
\setcounter{enumi}{1}
\item
$\frg=\frk \times \frr$ and $\frr$ is abelian.
\item
The ideal $\frr$ is definite.
\end{enumerate}
\end{prop}
\begin{proof}
 According to Theorem \ref{mthm:invariance}, 
nil-invariance implies that $\gs$ acts by skew derivations on
$\frg$ and on $\frg/\frg^{\perp}$. 
By assumption,  $\met$ induces a definite scalar product on the vector space 
$\frg/\frg^{\perp}$. Recall that a definite scalar product does not allow nil\-potent skew maps. 
Therefore, $[\frs + \frn, \frg] \subseteq \frg^\perp$. Similarly, for $X \in \frr$, $\ad(X)_{\rm n}(\frg) \subseteq \frg^{\perp}$ 
and thus also
$[\frr, \frr] \subseteq [\frr, \frn] +  \frg^\perp \subseteq  \frg^\perp$. 
Moreover, $[\frr,\frk \times \frs] = [\frn, \frk \times \frs] \subseteq \frg^\perp$. This shows (1), while (2) and (3) follow immediately, taking into account Theorem \ref{mthm:smallindex}.
\end{proof}

\subsection{Classification for relative index $\ell\leq 2$}

Now we specialize Theorem \ref{mthm:smallindex} to the two cases
$\ell=1$ and $\ell=2$ to obtain classifications of the Lie
algebras with nil-invariant symmetric bilinear forms in each
case.

{\renewcommand{\themthm}{\ref{mthm:lorentzian}}
\begin{mthm}
Let $\frg$ be a Lie algebra with nil-in\-variant symmetric bilinear
form $\met$ of relative index $\ell=1$,
and assume that $\frg^\perp$ does not contain a non-trivial ideal
of $\frg$.
Then one of the following cases occurs: 
\begin{enumerate}
\item[(I)]
$\frg= \frk \times \fra$, where $\fra$ is abelian and either
semidefinite or Lorentzian.
\item[(II)] $\frg=\frk \times\frr$, where $\frr$ is  Lorentzian of oscillator type.  
\item[(III)] $\frg=\frk \times\sl_2(\RR) \times \fra$, where 
$\fra$ is abelian and definite,  $\sl_2(\RR)$ is Lorentzian and 
$(\fra \times \frk) \perp \sl_2(\RR)$. 
\end{enumerate}
\end{mthm}
}
\begin{proof}
By Theorem \ref{mthm:smallindex}, $\frg^\perp \cap \frr = \zsp$.
Hence, $\frr$ is a subspace of index $\mu(\frr)\leq1$.

If $\frs \neq \zsp$, then by Lemma \ref{lem:sandr},
$\frs$ is non-degenerate, so that $\ell(\frs)=1$, as $\frs$ is
of non-compact type. Hence, $\frs = \sl_2(\RR)$.
Moreover, $\ell(\frr)=0$, that is, $\frr$ is definite and therefore abelian. The orthogonality is given by Lemma
\ref{lem:sandr}. This is case (III). 

Otherwise, $\frs = \zsp$. 
If $\frr$ is semi\-definite, then
$[\frr, \frr]\subseteq\frr^\perp\cap\frr$
by Proposition \ref{prop:riemann}.
By Lemma \ref{lem:sandr} (2), this implies 
$[\frr,\frr]\subseteq \frg^\perp\cap\frr=\zsp$.
Hence $\frr$ is abelian. This is the first part of case (I).
Assume $\frr$ is of Lorentzian type.
Then $\frr$ is non-degenerate since $\mu(\frr)\leq1$.
By the classification of invariant Lorentzian scalar products
(see remark following Example \ref{ex:oscillator1}),
$\frr$ is either abelian or contains a metric oscillator
algebra.
These are the second part of case (I) or case (II),
respectively.
\end{proof}

{\renewcommand{\themthm}{\ref{mthm:index2}}
\begin{mthm}
Let $\frg$ be a Lie algebra with nil-invariant symmetric bilinear 
form $\met$ of relative index $\ell=2$,
and assume that $\frg^\perp$ does not contain a non-trivial ideal
of $\frg$.
Then one of the following cases occurs:
\begin{enumerate}
\item[(I)]
$\frg=\frr\times\frk$, where $\frr$ is one of the following:
\begin{enumerate}
\item
$\frr$ is abelian.
\item
$\frr$ is Lorentzian of oscillator type.
\item
$\frr$ is solvable but non-abelian with invariant scalar product
of index $2$.
\end{enumerate}
\item[(II)]
$\frg=\fra\times\frk\times\frs$. Here, $\fra$ is abelian, 
$\frs = \sl_2(\RR) \times \sl_2(\RR)$ with a non-degenerate
invariant scalar product of index $2$.
Moreover, $\fra$ is definite and $(\fra \times \frk) \perp \frs$. 
\item[(III)]
$\frg =\frr\times\frk\times\sl_2(\RR)$,
where $\sl_2(\RR)$ is Lorentzian,
$(\frr\times\frk)\perp\sl_2(\RR)$,
and $\frr$ is one of the following:
\begin{enumerate}
\item
$\frr$ is abelian and either semidefinite or Lorentzian.
\item
$\frr$ is Lorentzian of oscillator type.
\end{enumerate}
\end{enumerate}
\end{mthm}
}
\begin{proof}
Write $s=\mu(\frs)$ and $r=\mu(\frr)$.
By Theorem \ref{mthm:smallindex}, $\frg^\perp \cap \gs = \zsp$. 
By Lemma \ref{lem:sandr}, this implies $s + r \leq 2$. 
Moreover, $\frs$ is non-degenerate and thus has index 
$s\leq\ell\leq2$.

First assume $s=0$, and therefore $\frs=\zsp$,
and  $r\leq 2$.
For $r=2$, the following possibilities arise:
$\frr$ is non-degenerate with relative index $\ell(\frr)=2$. 
This case falls into (I-a) or (I-c). 
Next, $\frr$ can be degenerate with $\ell(\frr)=1$,
in which case it is either abelian or of oscillator type,
the latter yielding part (I-b).
In the remaining case $\ell(\frr)=0$, $\frr$ is semidefinite. 
As in the proof of Theorem \ref{mthm:lorentzian}, this implies that  $\frr = \fra$ is abelian.
This completes case (I).

Assume $s=2$ and $r=0$. 
Then $\frs = \sl_2(\RR) \times \sl_2(\RR)$ and 
$\frr$ is definite and abelian.
The orthogonality is Lemma \ref{lem:sandr} (3).
This is case (II)

Now assume $s=1$ and $\frs = \sl_2(\RR)$,  $r \leq 1$. This yields the two possibilities for
$\frr$ in case (III).
\end{proof}

Note that the possible Lie algebras $\frr$ for case (I-c) of Theorem \ref{mthm:index2} above
are discussed in Section \ref{sec:solvable_invariant}.

\section{Further examples}
\label{sec:index3}

The examples in this section show that the properties of
nil-invariant symmetric bilinear forms with relative index
$\ell\leq 2$ given in Theorem \ref{mthm:smallindex} do not hold
for higher relative indices.
Let $\frk$, $\frs$, $\frr$ and $\gs$ be as in the previous
sections.\\

The following standard  construction for a Lie algebra with an
invariant scalar product  (\cf Medina \cite{medina}) shows that in general  $\frg$ does not have to be
a direct product of Lie algebras $\frk$, $\frs$ and $\frr$,
and that $\frs$ does not have to be orthogonal to $\frr$.

\begin{example}[Metric cotangent algebras]\label{ex:sl2sl2}
Let $\frg$ be a Lie algebra of dimension $n$, and
let $\ad^*$ denote the coadjoint representation of $\frg$ on
its dual vector space $\frg^*$, and consider the Lie
algebra $\hat \frg=\frg\ltimes_{\ad^*}\frg^*$.
The dual pairing defines an invariant scalar product on $\hat \frg$,
\[
\langle X_1+\xi_1, X_2+\xi_2 \rangle
=
\xi_1(X_2) + \xi_2(X_1),
\]
where $X_i\in\frg$ and $\xi_i\in\frg^*$.
The index of $\met$ is $n$.
We call such a $\frg$ a \emph{metric cotangent algebra}.
For example, if we choose $\frg=\sl_2(\RR)$,
then the index is $3$, 
and $\frk=\zsp$, $\frs=\sl_2(\RR)$ and $\hat \frg$ has abelian 
radical $\frr = \sl_2(\RR)^{*}  \cong\RR^3$.
In particular, $\frs$ is not orthogonal to $\frr$ and
$[\frs,\frr]\neq\zsp$.
\end{example}

The next example shows that for relative index $\ell=3$, the
transporter algebra $\frl$ of $\frb = \frr^\perp\cap\frr$ in $\frq = \frk$ 
(see Section \ref{sec:transporter_l2})  
can be trivial,
and as a consequence $\frg^\perp\cap\frr$ is not an ideal in $\frg$.
This contrasts the situation for $\ell \leq 2$, compare Corollary \ref{cor:isotrop_ideals_l2_0}.

\begin{example}\label{ex:so3so4}
Let $\frk=\so_3$, and
let $\frr=\so_3\oplus\so_3$, considered as a vector
space.
We write $\so_3^\scl$ and $\so_3^\scr$ to distinguish the two
summands of $\frr$, and for an element $X\in\so_3$, we write
$X^\scl=(X,0)\in\so_3^\scl$ and
$X^\scr=(0,X)\in\so_3^\scr$.
Let $\so_3^\triangle$ be the diagonal embedding of $\so_3$ in
$\frr$.

Let $T\in\frk$ act on $X=X_1^\scl+X_2^\scr\in\frr$ by
\[
\ad(T)X = [T,X_1]^\scl.
\tag{$*$}\label{*}
\]
This makes $\frr$ into a Lie algebra module for $\frk$, and we
can thus define a Lie algebra $\frg = \frk\ltimes\frr$ for
this action, taking $\frr$ as an abelian subalgebra. 
Observe also that $\so_3^\scr$ is the center of $\frg$.

Let $\kappa$ denote the Killing form on $\so_3$.
We define a symmetric bilinear form $\met$ on $\frg$ by
requiring
\[
\langle T,X_1^\scl+X_2^\scr\rangle = \kappa(T,X_1) - \kappa(T,X_2),
\quad
\frk\perp\frk,
\quad
\frr\perp\frr
\]
for all $T\in\frk$, $X_1^\scl+X_2^\scr\in\so_3$.
The adjoint operators of elements of $\frr$ are skew-symmetric
for $\met$. In fact, 
we have,  for all $X,Y\in\frr$, $Z\in\frg$,
\[
\langle [X,Y],Z\rangle
=0
=-\langle Y,[X,Z]\rangle
\]
and for $T,T'\in\frk$, by \eqref{*},
\begin{align*}
\langle [T,X],T'\rangle
&=\langle [T,X_1^\scl+X_2^\scr],T'\rangle
=\langle[T,X_1]^\scl,T'\rangle \\
&=\kappa([T,X_1],T')
=-\kappa(T,[T',X_1]) \\
&=-\langle T,[T',X_1]^\scl\rangle
=-\langle T,[T',X]\rangle.
\end{align*}
So $\met$ is indeed a nil-invariant form on $\frg$,
and, since $\frr^{\perp} = \frr$, 
\[
\frg^\perp=\frk^\perp\cap\frr=\so_3^\triangle
\quad\text{ and }\quad
\gs^\perp\cap\gs=\frr^\perp\cap\frr=\frr=\so_3^\scl\oplus\so_3^\scr.
\]
In particular, the index of $\met$ is $\mu=6$ and the
relative index is $\ell=3$.
Note that $\met$ is not invariant,
as $\frg^\perp$ is not an
ideal in $\frg$.
\end{example}

\begin{remark}
The construction in Example \ref{ex:so3so4} works if we replace
$\so_3$ by any other semisimple Lie algebra $\frf = \frk$ of compact type. 
However, if $\frf$ is not of compact type, then the resulting
bilinear form $\met$ will not be nil-invariant.
Geometrically this means that $\met$ cannot come from a
pseudo-Riemannian metric on a homogeneous space $G/H$
of finite volume, where $G$ is a Lie group with Lie algebra $\frg$.
\end{remark}


\section[Metric Lie algebras with abelian radical]{Metric Lie algebras with abelian radical}
\label{sec:abelianradical}

In this section we study finite-dimensional real Lie algebras
$\frg$ whose solvable radical $\frr$ is abelian and which are
equipped with a nil-invariant symmetric bilinear form $\met$.

\subsection{Abelian radical}
The Lie algebras with abelian radical and a nil-invariant
symmetric bilinear form decompose into three distinct types of
metric Lie algebras.

{
\renewcommand{\themthm}{\ref{mthm:abelian_rad1}}
\begin{mthm}
Let $\frg$ be a Lie algebra whose solvable radical $\frr$ is
abelian.
Suppose $\frg$ is equipped with a nil-invariant symmetric bilinear
form $\met$ such that the metric radical $\frg^\perp$ of $\met$
does not contain a non-trivial ideal of $\frg$.
Let $\frk\times\frs$ be a Levi subalgebra of $\frg$, where $\frk$
is of compact type and $\frs$ has no simple factors of compact
type.
Then $\frg$ is an orthogonal direct product of ideals
\[
\frg = \frg_1 \times \frg_2 \times \frg_3,
\]
with
\[
\frg_1=\frk\ltimes\fra,
\quad
\frg_2=\frs_0,
\quad
\frg_3=\frs_1\ltimes\frs_1^*,
\]
where $\frr=\fra\times\frs_1^*$ and $\frs=\frs_0\times\frs_1$ are
orthogonal direct products,
and $\frg_3$ is a metric cotangent algebra.
The restrictions of $\met$ to $\frg_2$ and $\frg_3$ are invariant
and non-degenerate.
In particular, $\frg^\perp\subseteq\frg_1$.
\end{mthm}
\addtocounter{mthm}{-1}
}

We split the proof into several lemmas.
Consider the submodules of invariants
$\frr^\frs, \frr^\frk\subseteq\frr$.
Since $\frs$, $\frk$ act reductively, we have
\[
[\frs,\frr]\oplus\frr^\frs=\frr=[\frk,\frr]\oplus\frr^\frk.
\]
Then $\fra=\frr^\frs$, $\frb=[\frs,\frr^\frk]$ and
$\frc=[\frs,\frr]\cap[\frk,\frr]$ are ideals in $\frg$ and
$\frr = \fra\oplus\frb\oplus\frc$.
Recall from Theorem \ref{mthm:invariance} that $\met$ is
in particular $\frs$- and $\frr$-invariant.

\begin{lem}\label{lem:R=AxB}
$\frc=\zsp$ and $\frr$ is an orthogonal direct sum of ideals in $\frg$
\[
\frr = \fra\times\frb
\]
where $[\frk,\frr]\subseteq\fra$ and $[\frs,\frr]=\frb$.
\end{lem}
\begin{proof}
The $\frs$-invariance of $\met$ immediately implies $\fra\perp\frb$.
Since $\frr$ is abelian, $\frr$-invariance implies $\frc\perp\frr$.
Since $\frc\perp(\frs\times\frk)$ by Corollary \ref{cor:ksperp},
this shows $\frc$ is an ideal contained in
$\frg^\perp$, hence $\frc=\zsp$.
Now $[\frk,\frr]\subseteq\fra$ and $[\frs,\frr]=\frb$ by
definition of $\fra$ and $\frb$.
\end{proof}

\begin{lem}\label{lem:G=KAxSB}
$\frg$ is a direct product of ideals
\[
\frg = (\frk\ltimes\fra)\times(\frs\ltimes\frb),
\]
where $(\frk\ltimes\fra)\perp(\frs\ltimes\frb)$.
\end{lem}
\begin{proof}
The splitting as a direct product of ideals follows from Lemma \ref{lem:R=AxB}.
The orthogonality follows together with Corollary \ref{cor:ksperp} and
the fact that the $\frs$-invariance of $\met$ implies $\frs\perp\fra$
and $\frk\perp\frb$.
\end{proof}

\begin{lem}\label{lem:SB}
$\frg^\perp\subseteq \frk\ltimes\fra$
and
$\frs\ltimes\frb$ is a non-degenerate ideal of $\frg$.
\end{lem}
\begin{proof}
$\zen(\gs)=\fra$, therefore $\frg^\perp\subseteq\frk\ltimes\fra$
by Corollary \ref{mcor:invariance_radical}.
Since also $(\frs\ltimes\frb)\perp(\frk\ltimes\fra)$, we have
$(\frs\ltimes\frb)\cap(\frs\ltimes\frb)^\perp\subseteq\frg^\perp\subseteq\frk\ltimes\fra$. It follows that
$(\frs\ltimes\frb)\cap(\frs\ltimes\frb)^\perp=\zsp$.
\end{proof}


To complete the proof of Theorem \ref{mthm:abelian_rad1}, it remains
to understand the structure of the ideal $\frs\ltimes\frb$, which by
Theorem \ref{mthm:invariance} and the preceding lemmas is a
Lie algebra with an invariant non-degenerate scalar product given by
the restriction of $\met$.

\begin{lem}\label{lem:trivial}
$\frb$ is totally isotropic.
Let $\frs_0$ be the kernel of the $\frs$-action on $\frb$.
Then $\frs_0=\frb^\perp\cap\frs$.
\end{lem}
\begin{proof}
Since $\met$ is $\frr$-invariant and $\frr$ is abelian,
$\frb$ is totally isotropic.
For the second claim, use $\frb\cap\frs^\perp=\zsp$ and the invariance
of $\met$.
\end{proof}

\begin{lem}\label{lem:coadjoint}
$\frs$ is an orthogonal direct product of ideals
$\frs=\frs_0\times\frs_1$
with the following properties:
\begin{enumerate}
\item
$\frs_1\ltimes\frb$ is a metric cotangent algebra.
\item
$[\frs_0,\frb]=\zsp$ and $\frs_0=\frb^\perp\cap\frs$.
\end{enumerate}
\end{lem}
\begin{proof}
The kernel $\frs_0$ of the $\frs$-action on $\frb$ is an ideal in
$\frs$, and by Lemma \ref{lem:trivial}
orthogonal to $\frb$.
Let $\frs_1$ be the ideal in $\frs$ such that $\frs=\frs_0\times\frs_1$.
Then $\frs_0\perp\frs_1$ by invariance of $\met$.

$\frs_1$ acts faithfully on $\frb$ and so $\frs_1\cap\frb^\perp=\zsp$
by Lemma \ref{lem:trivial}.
Moreover, $\frs_1\ltimes\frb$ is non-degenerate since $\frs\ltimes\frb$
is. But $\frb$ is totally isotropic by Lemma \ref{lem:trivial}, so
non-degeneracy implies $\dim\frs_1=\dim\frb$.
Therefore $\frb$ and $\frs_1$ are dually paired by $\met$.

Now identify $\frb$ with $\frs_1^*$ and write $\xi(s)=\langle \xi,s\rangle$
for $\xi\in\frs_1^*$, $S\in\frs_1$.
Then, once more by invariance of $\met$,
\[
[S,\xi](S') = \langle [S,\xi],S'\rangle = \langle \xi,-[S,S']\rangle
= \xi(-\ad(S)S') = (\ad^*(S)\xi)(S')
\]
for all $S,S'\in\frs_1$. So the action of $\frs_1$ on $\frs_1^*$
is the coadjoint action.
This means $\frs\ltimes\frb$ is a metric cotangent algebra
(cf.~Example \ref{ex:sl2sl2}).
\end{proof}

\begin{proof}[Proof of Theorem \ref{mthm:abelian_rad1}]
The decomposition into the desired orthogonal ideals follows from
Lemmas \ref{lem:G=KAxSB} to \ref{lem:coadjoint}.
The structure of the ideals $\frg_2$ and $\frg_3$ is Lemma
\ref{lem:coadjoint}.
\end{proof}

The algebra $\frg_1$ in Theorem \ref{mthm:abelian_rad1} is of
Euclidean type.
Let $\frg=\frk\ltimes V$, with $V\cong\RR^n$, be an algebra of
Euclidean type and let $\frk_0$ be the kernel of the $\frk$-action
on $V$.
Proposition \ref{prop:gperp_inv} and the fact that the solvable
radical $V$ is abelian immediately give the following:

\begin{prop}\label{prop:Hcompact}
Let $\frg=\frk\ltimes V$ be a Lie algebra of Euclidean type,
and suppose $\frg$ is equipped with a symmetric 
bilinear form that is nil-invariant and $\frg^\perp$-invariant,
such that $\frg^\perp$ does not contain a non-trivial ideal of
$\frg$.
Then
\begin{equation}
\frg^\perp \subseteq \frk_0\times V.
\label{eq:GperpK0V0V1}
\end{equation}
\end{prop}

The following Examples \ref{ex:not_easy} and \ref{ex:not_easy2}
show that in general a metric Lie algebra of Euclidean type cannot
be further decomposed into orthogonal direct sums of metric Lie
algebras.
Both examples will play a role in Section \ref{sec:simplyconn}.

\begin{example}\label{ex:not_easy}
Let $\frk_1=\so_3$, $V_1=\RR^3$, $V_0=\RR^3$
and
$\frg=(\so_3\ltimes V_1)\times V_0$ with the natural action of
$\so_3$ on $V_1$.
Let $\varphi:V_1\to V_0$ be an isomorphism and
put
\[
\frh = \{(0,v,\varphi(v))\mid v\in V_0\}\subset(\frk_0\ltimes V_1) \times V_0.
\]
We can define a nil-invariant symmetric bilinear form on $\frg$
by identifying $V_1\cong\so_3^*$ and requiring
for $K\in\frk_1$, $v_0\in V_0$, $v_1\in V_1$,
\[
\langle K,v_0+v_1\rangle
=
v_1(K)-\varphi^{-1}(v_0)(K),
\]
and further $\frk_1\perp\frk_1$,
$(V_0\oplus V_1)\perp(V_0\oplus V_1)$.
Then $\met$ has signature $(3,3,3)$
and metric radical $\frh=\frg^\perp$, which is not an ideal in $\frg$.
Note that $\met$ is not invariant.
Moreover, $\frk_1\ltimes V_1$ is not orthogonal to $V_0$.
A direct factor $\frk_0$ can be added to this example at liberty.
\end{example}

\begin{example}\label{ex:not_easy2}
We can modify the Lie algebra $\frg$ from Example \ref{ex:not_easy} 
by embedding
the direct summand $V_0\cong\RR^3$ in a torus subalgebra in a semisimple Lie
algebra $\frk_0$ of compact type, say $\frk_0=\so_6$, to obtain a
Lie algebra $\frf=(\frk_1\ltimes V_1)\times\frk_0$.
We obtain a nil-invariant symmetric bilinear form of signature
$(15,3,3)$ on $\frf$ by extending $\met$ by a definite form on a
vector space complement of $V_0$ in $\frk_0$. The metric radical of the
new form is still $\frg^\perp=\frh$.
\end{example}

\subsection{Nil-invariant bilinear forms on Euclidean algebras}

A \emph{Euclidean algebra} is a Lie algebra
$\euc_n=\so_n\ltimes\RR^n$,
where $\so_n$ acts on $\RR^n$ by the natural action.

%

\begin{example}\label{ex:so3R3}
Consider $\frg=\so_3\ltimes \RR^n$ with a nil-invariant symmetric bilinear
form $\met$, and assume that the action of $\so_3$ is irreducible.
By Proposition \ref{prop:so3_skew_module}, either
$\so_3\perp\RR^n$, 
or $n=3$ and $\so_3$ acts by its coadjoint representation on
$\RR^3\cong\so_3^*$, and $\met$ is the dual pairing.
In the first case, $\RR^n$ is an ideal in $\frg^\perp$, and in the
second case, $\met$ is invariant and non-degenerate.
\end{example}

\begin{example}\label{ex:so4R4}
Let $\frg$ be the Euclidean Lie algebra $\so_4\ltimes \RR^4$ with a
nil-invariant symmetric bilinear form $\met$.
Since $\so_4\cong\so_3\times\so_3$, and here both factors
$\so_3$ act irreducibly on $\RR^4$, it follows from
Example \ref{ex:so3R3} that in $\frg$, $\RR^4$ is orthogonal to both
factors $\so_3$, hence to all of $\so_4$.
In particular, $\RR^4$ is an ideal contained in $\frg^\perp$.
\end{example}

\begin{thm}\label{thm:noSOnRn}
Let $\met$ be a nil-invariant symmetric bilinear form on the
Euclidean Lie algebra $\so_n\ltimes\RR^n$ for $n\geq 4$.
Then the ideal $\RR^n$ is contained in $\frg^\perp$.
\end{thm}
\begin{proof}
For $n=4$, this is Example \ref{ex:so4R4}. So assume $n>4$.
Consider an embedding of $\so_4$ in $\so_n$ such that
$\RR^n=\RR^4\oplus\RR^{n-4}$, where $\so_4$ acts on $\RR^4$ in the
standard way and trivially on $\RR^{n-4}$.
By Example \ref{ex:so4R4}, $\so_4\perp\RR^4$.
Since $\RR^{n-4}\subseteq[\so_n,\RR^n]$, the nil-invariance of
$\met$ implies $\so_4\perp\RR^{n-4}$.
Hence $\RR^n\perp\so_4$.

The same reasoning shows that $\Ad(g)\so_4\perp\RR^n$, where
$g\in\SO_n$.
Then $\frb = \sum_{g\in\SO_n}\Ad(g)\so_4$ is orthogonal to $\RR^n$.
But $\frb$ is clearly an ideal in
$\so_n$, so $\frb=\so_n$ by simplicity of $\so_n$ for $n>4$.
\end{proof}




{
\renewcommand{\themthm}{\ref{thm:noSOnRn2}}
\begin{mthm}
The Euclidean group $\Euc_n=\OO_n\ltimes\RR^n$, $n\neq 1, 3$,
does not have
compact quotients with a pseudo-Riemannian metric such that
$\Euc_n$ acts isometrically and almost effectively.
\end{mthm}
\addtocounter{mthm}{-1}
}
\begin{proof}
For $n>3$, such an action of $\Euc_n$ would induce
a nil-invariant symmetric bilinear form on the Lie algebra
$\so_n\ltimes\RR^n$ without non-trivial ideals in its metric radical,
contradicting Theorem \ref{thm:noSOnRn}.

For $n=2$, the Lie algebra $\euc_2$ is solvable, and hence by
Baues and Globke \cite{BG}, any nil-invariant symmetric bilinear
form must be invariant. For such a form, the ideal $\RR^2$ of $\euc_2$
must be contained in $\euc_2^\perp$, and therefore the action cannot be
effective.

Note that $\euc_3$ is an exception, as it is the metric cotangent
algebra of $\so_3$.
\end{proof}

\begin{remark}
Clearly the Lie group $\Euc_n$ admits compact quotient manifolds on
which $\Euc_n$ acts almost effectively.
For example take the quotient by a subgroup $F\ltimes\ZZ^n$, where
$F\subset\OO_n$ is a finite subgroup preserving $\ZZ^n$.
Compact quotients with finite fundamental group also exist.
For example, 
for any non-trivial homomorphism $\varphi:\RR^n\to\OO_n$, the graph
$H$ of $\varphi$ is a closed subgroup of $\Euc_n$ isomorphic
to $\RR^n$, and the quotient $M=\Euc_n/H$ is compact (and diffeomorphic
to $\OO_n$).
Since $H$ contains no non-trivial normal subgroup of $\Euc_n$,
the $\Euc_n$-action on $M$ is effective.
Theorem \ref{thm:noSOnRn2} tells us that none of these quotients
admit $\Euc_n$-invariant pseudo-Riemannian metrics.
\end{remark}


\section[Simply connected spaces]{Simply connected compact homogeneous spaces with indefinite metric}
\label{sec:simplyconn}

Let $M$ be a connected and simply connected pseudo-Riemannian
homogeneous space of finite volume.
Then we can write
\begin{equation}
M=G/H
\label{eq:MGH}
\end{equation}
for a connected Lie group $G$ acting almost
effectively and by isometries on $M$, and $H$ is a closed subgroup
of $G$ that contains no non-trivial connected normal subgroup of $G$
(for example, $G=\Iso(M)^\circ$).
Note that $H$ is connected since $M$ is simply connected.


We decompose $G=KSR$, where $K$ is a compact
semi\-simple subgroup, $S$ is a semisimple subgroup without compact factors, $R$ the solvable radical of $G$

\begin{prop}\label{prop:S_trivial_M_compact}
The subgroup $S$ is trivial and $M$ is compact. 
\end{prop}
\begin{proof}
As $M$ is simply connected, $H=H^\circ$.
Now $H\subseteq K R$ by Corollary \ref{mcor:invariance_radical}.
On the other hand, since $M$ has finite invariant volume, the Zariski
closure of $\Ad_\frg(H)$ contains $\Ad_\frg(S)$, see
Mostow \cite[Lemma 3.1]{mostow3}.
Therefore $S$ must be trivial.
It follows from Mostow's result \cite[Theorem 6.2]{mostow4} on quotients of
solvable Lie groups that $M=(KR)/H$ is compact.
\end{proof}

We can therefore restrict ourselves in \eqref{eq:MGH} to groups
$G = K R$
and connected uniform subgroups $H$ of $G$.

The structure of a general compact homogeneous manifold with finite
fundamental group is surveyed in Onishchik and Vinberg \cite[II.5.\S 2]{OV1}.
In our context it follows that
\begin{equation}
G = L\ltimes V
\label{eq:OV}
\end{equation}
where $V$ is a normal subgroup isomorphic to $\RR^n$ and
$L=KA$ is a maximal compact subgroup of $G$.
The solvable radical is $R=A\ltimes V$. Moreover, $V^L=\zsp$.
By a theorem of Montgomery \cite{montgomery} (also \cite[p.~137]{OV1}), $K$ acts transitively on $M$.
\\

The existence of a $G$-invariant metric on $M$ further restricts the
structure of $G$.

\begin{prop}\label{prop:rad_abelian}
The solvable radical $R$ of $G$ is abelian.
In particular, $R=A\times V$, $V^K=\zsp$
and $A=\Zen(G)^\circ$.
\end{prop}
\begin{proof}
Let $\Zen(R)$ denote the center of $R$ and $N$ its nilradical.
Since $H$ is connected, $H\subseteq K\Zen(R)^\circ$ by Corollary
\ref{mcor:invariance_radical}.
Hence there is a surjection
$G/H\to G/(K\Zen(R)^\circ)=R/\Zen(R)^\circ$.
It follows that $\Zen(R)^\circ$ is a 
connected uniform subgroup. Therefore the nilradical $N$ of $R$ 
is $N=T\Zen(R)^\circ$ for some compact torus $T$.
But a compact subgroup of $N$ must be central in $R$,
so $T\subseteq\Zen(R)$.
Hence $N\subseteq\Zen(R)$, which means $R=N$ is abelian.
\end{proof}

Combined with \eqref{eq:OV}, we obtain
\begin{equation}
G = K R = (K_0 A)\times(K_1\ltimes V),
\label{eq:GKR}
\end{equation}
with $K=K_0\times K_1$, $R=A\times V$, where
$K_0$ is the kernel of the $K$-action on $V$.\\

%


For any subgroup $Q$ of $G$ we write $H_{Q}=H\cap Q$.

\begin{lem}\label{lem:HKHLnormal}
$[H,H]\subseteq H_K$.
In particular, $H_K$ is a normal subgroup of $H$.
\end{lem}
\begin{proof}
By Proposition \ref{prop:Hcompact} and the connectedness of $H$,
we have $H\subseteq K_0 R$.
Since $R$ is abelian, $[H,H]\subseteq H_{K_0}$.
\end{proof}

If $G$ is simply connected, we have $K\cap R=\one$. Then let
$\p_K$, $\p_R$ denote the projection maps from $G$ to $K$, $R$.

\begin{lem}\label{lem:surjective_R}
Suppose $G$ is simply connected.
Then $\p_R(H)=R$.
\end{lem}
\begin{proof}
Since $K$ acts transitively on $M$, we have $G=KH$.
Then $R=\p_R(G)=\p_R(H)$.
\end{proof}

\begin{prop}\label{prop:HKE}
Suppose $G$ is simply connected.
Then the stabilizer $H$ is a semidirect product $H=H_K\times E$, where $E$ is the graph of a homomorphism $\varphi:R\to K$ that is non-trivial if $\dim R>0$.
Moreover, $\varphi(R\cap H)=\one$.
\end{prop}
\begin{proof}
The subgroup $H_{K}$ is a maximal compact subgroup of the stabilizer
$H$. By Lemma \ref{lem:HKHLnormal},
$H=H_K\times E$ for some normal subgroup $E$
diffeomorphic to a vector space.
By Lemma \ref{lem:surjective_R}, $H$ projects onto $R$ with kernel
$H_K$, so that $E\cong R$.
Then $E$ is necessarily the graph of a homomorphism $\varphi:R\to K$.
If $\dim R>0$, then $\varphi$ is non-trivial, for otherwise
$R\subseteq H$, in contradiction to the almost effectivity of the
action.
Observe that $R\cap H\subseteq E$.
Therefore $\varphi(R\cap H)\subseteq H_K\cap E=\one$.
\end{proof}

Now we can state our main result:

{\renewcommand{\themthm}{\ref{mthm:geometric}}
\begin{mthm}
Let $M$ be a connected and simply connected pseudo-Riemannian
homogeneous space of finite volume, $G=\Iso(M)^\circ$,
and let $H$ be the stabilizer subgroup in $G$ of a point in $M$.
Let $G=KR$ be a Levi decomposition, where $R$ is the solvable radical
of $G$.
Then:
\begin{enumerate}
\item
$M$ is compact.
\item
$K$ is compact and acts transitively on $M$.
\item
$R$ is abelian.
Let $A$ be the maximal compact subgroup of $R$. Then $A=\Zen(G)^\circ$.
More explicitely, $R=A\times V$ where $V\cong\RR^n$ and $V^{K}=\zsp$.
\item
$H$ is connected.
If $\dim R>0$, then $H=(H\cap K) E$, where $E$ and $H\cap K$ are
normal subgroups in $H$, $(H\cap K)\cap E$ is finite,
and $E$ is the graph of a non-trivial homomorphism
$\varphi:R\to K$, where the restriction $\varphi|_A$ is injective.
\end{enumerate}
\end{mthm}
\addtocounter{mthm}{-1}
}
\begin{proof}
Claims (1), (2) and (3) follow from
Proposition \ref{prop:S_trivial_M_compact},
Proposition \ref{prop:rad_abelian}
and \eqref{eq:OV},
applied to $G=\Iso(M)^\circ$.

For claim (4), let $\tilde{G}$ be the universal cover of $G$.
Since $G$ acts effectively on $M$, $\tilde{G}$ acts almost effectively
on $M$ with stabilizer $\tilde{H}$, the preimage of $H$ in $\tilde{G}$.
Let $\tilde{\varphi}:\tilde{R}\to\tilde{K}$ be the homomorphism given by
Proposition \ref{prop:HKE} for $\tilde{G}$.
Then $\tilde{R}=\tilde{A}\oplus V$,
with $\tilde{A}\cong\RR^k$ for some $k$, and $R=\tilde{R}/Z$ for some
central discrete subgroup $Z\subset\tilde{A}\cap\tilde{H}$.
Since $Z\subset\tilde{R}\cap\tilde{H}$ we have
$Z\subseteq\ker\tilde{\varphi}$.
Therefore $\tilde{\varphi}$ descends to a homomorphism
$R\to\tilde{K}$, and by composing with the canonical projection
$\tilde{K}\to K$,
we obtain a homomorphism $\varphi:R\to K$ with the desired properties.
Observe that $\ker\varphi|_A\subset A\cap H$ is a normal subgroup in $G$.
Hence it is trivial, as $G$ acts effectively.
\end{proof}

Now assume further
that the index of the metric on $M$ is $\ell\leq 2$.
Theorem \ref{mthm:smallindex} has strong consequences in the simply
connected case.

{\renewcommand{\themthm}{\ref{mthm:sc_index2}}
\begin{mthm}
The isometry group of any simply connected pseudo-Riemannian
homogeneous manifold of finite volume and metric index $\ell\leq2$
is compact.
\end{mthm}
}
\begin{proof}
We know from Theorem \ref{mthm:geometric} that $M$ is compact.
Let $G=\Iso(M)^\circ$, with $G=KR$ as before.
By Theorem \ref{mthm:smallindex}, $R$ commutes with $K$
and thus $R=A$ by part 3 of Theorem \ref{mthm:geometric}.
It follows that $G=K A$ is compact.

Then $K$ is a characteristic subgroup of $G$ which also acts
transitively on $M$.
Therefore we may identify $\mathrm{T}_x M$ at $x\in M$ with
$\frk/(\frh\cap\frk)$, where $\frk$ is the Lie algebra of $K$.
Hence the isotropy representation of the stabilizer $\Iso(M)_x$
factorizes over a closed subgroup of the automorphism group of $\frk$.
As this latter group is compact, the isotropy representation
has compact closure in $\GL(\mathrm{T}_x M)$.
If follows that there exists a
Riemannian metric on $M$ that is preserved by $\Iso(M)$.
Hence $\Iso(M)$ is compact.
\end{proof}

\begin{remark}
Note that in fact the isometry group of every compact analytic simply
connected pseudo-Riemannian manifold has finitely
many connected components (Gromov \cite[Theorem 3.5.C]{gromov}).
\end{remark}


For indices higher than two,
the identity component of the iso\-metry group of a
simply connected $M$ can be non-compact.
This is demonstrated by the examples below in which we construct some
$M$ on which a non-compact group acts isometrically.
The following Lemma \ref{lem:no_compact} then ensures that these
groups cannot be contained in any compact Lie group.

\begin{lem}\label{lem:no_compact}
Assume that the action of $K$ on $V$ in the semidirect product
$G=K\ltimes V$ is non-trivial.
Then $G$ cannot be immersed in a compact Lie group.
\end{lem}
\begin{proof}
Suppose there is a compact Lie group $C$ that contains $G$ as a
subgroup.
As the action of $K$ on $V$ is non-trivial, there exists an
element $v\in V\subseteq C$ such that $\Ad_{\frc}(v)$ has non-trivial
unipotent Jordan part. But by compactness of $C$, every $\Ad_{\frc}(g)$,
$g\in C$, is semisimple, a contradiction.
\end{proof}

\begin{example}\label{ex:not_easy_group}
Start with
$G_1=(\tilde{\SO}_3\ltimes\RR^3)\times\TT^3$,
where $\tilde{\SO}_3$ acts on $\RR^3$ by the coadjoint action,
and let $\varphi:\RR^3\to\TT^3$
be a homomorphism with discrete kernel.
Put
\[
H = \{(\I_3,v,\varphi(v))\mid v\in\RR^3\}.
\]
The Lie algebras $\frg_1$ of $G_1$ and $\frh$ of $H$ are the
corresponding Lie algebras from Example \ref{ex:not_easy}.
We can extend the nil-invariant scalar product $\met$ on $\frg_1$
from Example \ref{ex:not_easy} to a left-invariant tensor on $G_1$,
and thus obtain a $G_1$-invariant pseudo-Riemannian metric of signature
$(3,3)$ on the quotient $M_1=G_1/H=\tilde{\SO}_3\times\TT^3$.
Here, $M_1$ is a non-simply connected manifold with a non-compact
connected stabilizer.

In order to obtain a simply connected space, embed $\TT^3$ in a simply
connected compact semisimple group $K_0$, for example
$K_0=\tilde{\SO}_6$, so that $G_1$ is embedded in
$G=(\tilde{\SO}_3 \ltimes \RR^3) \times K_0$.
As in Example \ref{ex:not_easy2}, we can extend $\met$ from $\frg_1$
to $\frg$, and thus obtain a compact simply connected pseudo-Riemannian
homogeneous manifold $M=G/H=\tilde{\SO}_3\times K_0$.
%
%
\end{example}

\begin{example}
Example \ref{ex:not_easy_group} can be generalized by replacing
$\tilde{\SO}_3$ by any simply connected compact semi\-simple group
$K$, acting by the coadjoint representation on $\RR^d$, where
$d=\dim K$, and trivially on $\TT^d$.
Define $H$ similarly as a graph in $\RR^d\times\TT^d$,
and embed $\TT^d$ in a simply connected compact semisimple Lie group
$K_0$.
\end{example}

\appendix


\section{Modules with skew pairings}
\label{sec:skewmodule}

Let $\frg$ be a  Lie algebra and let $V$ be a
finite-dimensional $\frg$-module. Here we work over a fixed ground field 
$\kk$ of characteristic $0$.

\begin{definition}
A bilinear map $\met: \frg\times V \to \kk$ such that
for all $X,Y\in\frg$, $v\in V$,
\begin{equation}
\langle X,  Y v \rangle = -\langle Y, X v\rangle
\label{eq:skewpairing}
\end{equation}
is called a \emph{skew pairing} for $V$,
and $V$ is called a
\emph{skew module} for $\frg$.
\end{definition}


We make the following elementary observations:

\begin{lem}\label{lem:skew_basic}
Assume that there exists $X \in \frg$ such that the map $v \mapsto X v$, $v \in V$,  is an
invertible linear operator of\/ $V$.
Then every skew pairing for $V$ is zero. 
More generally, let $X, Y \in \frg$ and $W \subseteq V$ such that $Y  W  
\subseteq X V$.
Then $Y \perp X W$. 
\end{lem}
\begin{proof}
Let $w \in W$ and $v \in V$ with  $Y w =  X v$. 
Then 
\[  \langle Y, X w \rangle
= -\langle X, Y w \rangle
= -\langle X, X  v  \rangle=0.
\qedhere
\]
\end{proof} 

\begin{lem}\label{lem:skew_basic2}
If $X \perp V$ then $\frg \perp X V$.
\end{lem}
\begin{proof}  Let $Y \in \frg$ and $v \in V$.  
Then 
$  \langle Y, X v \rangle
= -\langle X, Y v \rangle =0$.
\end{proof}

\subsection{Skew pairings for $\boldsymbol{\sl_2(\kk)}$}

The following determines all skew pairings for the Lie algebra $\frg=\sl_2(\kk)$.


\begin{prop}\label{prop:sl2_skew_module}
Let $\met: \sl_2(\kk) \times V \to \kk$ be a skew pairing for the (non-trivial) irreducible module $V$.
If the skew pairing is non-zero, then $V$ is isomorphic to the adjoint representation
of  $\sl_2(\kk)$ and $\met$ is proportional to the Killing form. 
\end{prop}
\begin{proof}  We choose a standard basis $X,Y,H$ for $\sl_2(\kk)$
such that 
$[X,Y] = H$, $[H,X] = 2X$,  $[H,Y] =  -2Y$. Let $V = \kk^2$ denote the standard 
representation. Let $e_1, e_2$ be a basis such that $X e_1 =  0$, $X e_2 = e_1$.
Since $H e_1 = e_1$ and $H e_2 =  - e_2$,  the operator defined by $H$ is invertible.
Hence, it follows that every skew pairing for $V$ is zero by Lemma \ref{lem:skew_basic}.

The irreducible modules for $\sl_2(\kk)$ are precisely the symmetric powers $V_k = \SS^{k} V$, $k \geq 1$. 
Note that,  in $V_k$, $\im X$ is spanned by the product vectors $e_1^\ell e_2^{k -\ell}$, $k \geq \ell \geq 1$. Similarly, $\im Y$ is 
spanned by  $e_1^{k -\ell} e_2^\ell$,  $k \geq \ell \geq 1$. 

Consider $W$,  the subspace of
$\im Y$ spanned by $e_1^{k -\ell} e_2^\ell$,  $\ell \geq 2$. Now $X W \subseteq \im Y$ and
from Lemma \ref{lem:skew_basic} we can conclude that $X \perp Y W $. Observe that $Y W$ is
spanned by $e_1^{k -\ell-1}  e_2^{\ell +1}$,  $\ell \geq 2$, $k \geq \ell +1$.  
In particular, $X \perp e_2^k$ if $k \geq 3$. Since also 
$X \perp \im X$, this shows $X \perp V_k$, $k \geq 3$.
By symmetry, we also see that $Y \perp V_k$, $k \geq 3$. 
Since $\im X$, $\im Y$ together span $V_k$, we conclude (using Lemma \ref{lem:skew_basic2}) 
that $\sl_2(\kk) \perp V_k$, $k \geq 3$.

Finally, the module $V_2$ is isomorphic to the adjoint representation. Consider the 
Killing form $\kappa \!: \sl_2(\kk) \times \sl_2(\kk) \to \kk$. Recall that  $\kappa$ is symmetric and skew 
with respect to the adjoint representation of $\sl_2(\kk)$ on itself. Therefore, it also defines a 
skew pairing for $V_2$.  An evident computation using the skew-condition on commutators in $\sl_2(\kk)$ shows
that every skew form $\met$ for the adjoint representation is determined by its value $\langle H , H \rangle$. 
Hence, it must be proportional to the Killing form. 
\end{proof}


\subsection{Application to $\boldsymbol{\so_3}$ over the reals}
Here we consider the simple Lie algebra $\so_3$ over the real numbers. 
Since  $\so_3$ has  complexification $\sl_2(\CC)$, we can apply 
Proposition \ref{prop:sl2_skew_module} to show:

\begin{prop}\label{prop:so3_skew_module}
Let $\met: \so_3 \times V \to \RR$ be a skew pairing for the (non-trivial) irreducible module $V$.
If the skew pairing is non-zero, then $V$ is isomorphic to the adjoint representation
of  $\so_3$ and $\met$ is proportional to the Killing form. 
\end{prop}
\begin{proof}
Using the isomorphism of  $\so_3$ with the Lie algebra $\su_2$, we view $\so_3$ as
a subalgebra of $\sl_2(\CC)$. We thus see that the irreducible complex representations of $\so_3$ 
are precisely the $\su_2$-modules $\SS^{k} \CC^2$. 

Now let $V$ be a real module for $ \so_3$, which is irreducible and non-trivial, and assume
that $\met$ is a non-trivial skew pairing for $V$. We may extend 
$V$ to a complex linear skew pairing
$\met_\CC  \!:\sl_2(\CC) \times V_\CC \to \CC$, where
$V_\CC$ denotes complexification of the $\su_2$-module $V$. 

In case $V_\CC$ is an irreducible module for  $\sl_2(\CC)$, Proposition \ref{prop:sl2_skew_module} shows that 
$V_\CC = V_2$ is the adjoint representation of  $\sl_2(\CC)$.
Hence, $V$ must  have been the adjoint representation of
$\so_3$. 
 
Otherwise, if $V_\CC$ is reducible, $V$ is one of the modules $V_k = \SS^{2 \ell -1}\CC^2$ with scalars restricted to the reals
(cf.~Br\"ocker and tom Dieck \cite[Proposition 6.6]{BT}).  It also follows that  $V_\CC$ is isomorphic to a direct sum of
$\SS^{2 \ell -1} \CC^2 $ with itself. 
Since we assume that the skew pairing $\met_\CC$ for $V_\CC$ is non-trivial, Proposition \ref{prop:sl2_skew_module}
implies that one of the irreducible summands of $V_\CC$ is isomorphic to $\SS^{2} \CC^2$. This is impossible, since $k = 2 \ell -1 $ is odd.
\end{proof}

The Killing form is always a non-degenerate pairing. In the light of the previous two propositions, this give us: 

\begin{cor}\label{cor:skewkernel} 
Let $\met: \frg  \times V \to \kk$  be a skew pairing,
where either $\frg = \sl_2(\kk)$ or $\frg = \so_3$ and 
$\kk = \RR$. Assume further that 
$V^{\frg} = \{ v \in V \mid \frg v = \zsp \}   = \zsp$.
Define
\[
V^{\perp} = \{ X \in \frg \mid \langle X, V \rangle = \zsp \} \; .
\]
Then either $V^{\perp} = \zsp$ or\/  $V^{\perp} = \frg$. 
\end{cor} 
\begin{proof}
The first case occurs precisely if there exists an irreducible summand $W$ of $V$ on which 
the restricted skew pairing $\frg \times W \to \kk$ induces the Killing form. 
\end{proof}

%


\section{Nil-invariant scalar products on solvable Lie algebras}
\label{sec:newproof}

We present a new proof for a key result
of Baues and Globke \cite[Theorem 1.2]{BG}.
The importance of this result lies in it being the crucial
ingredient in the proof of our Theorem \ref{mthm:invariance},
which supersedes it and is itself the fundamental tool in the study
of Lie algebras with nil-invariant bilinear forms.

\begin{thm}[Baues \& Globke]\label{thm:BG}
Let $\frg$ be a finite-dimensional solvable real Lie algebra, and $\met$ a nil-invariant symmetric bilinear form on
$\frg$. Then $\met$ is invariant.
\end{thm}
%

We recall some well-known facts (see Jacobson \cite[Chapter III]{jacobson}).
Let $\frg$ be an arbitrary finite-dimensional real Lie algebra.
For $X\in\frg$ let $\frg(X,0)$ denote the maxi\-mal
subspace of $\frg$ on which $\ad(X)$ is nilpotent.
Let $H_0$ be a regular element of $\frg$, that is,
$\dim\frg(H_0,0) = \min\{\dim\frg(X,0)\mid X\in\frg\}$.
We write $\frg_0=\frg(H_0,0)$ for short.
Then, by \cite[Chapter III, Theorem 1.1, Proposition 1.1]{jacobson},
$\frg_0$ is a Cartan subalgebra of $\frg$, and there is a Fitting
decomposition
\[
\frg=\frg_0\oplus\frg_1
\]
into $\frg_0$-submodules.
In particular, as a Cartan subalgebra, $\frg_0$ is nilpotent, and the
restriction of $\ad(H_0)$ to $\frg_1$ is an isomorphism.

\begin{lem}\label{lem:regular}
Any $X\in\frg_0$ sufficiently close to $H_0$ in $\frg_0$ is also regular,
and then
\[
\frg(X,0)=\frg(H_0,0) =  \frg_0\; .
\]
\end{lem}
\begin{proof}
The set of regular elements in $\frg$ is Zariski-open, and thus intersects
$\frg_0$ in a non-empty Zariski-open set (it contains $H_0$).
So any $X\in\frg_0$ sufficiently close to $H_0$ is also a regular element.
Two Cartan subalgebras with a common regular element coincide
\cite[p.~60]{jacobson}, so that $\frg_0=\frg(X,0)$.
\end{proof}

\begin{lem}\label{lem:adh_invariant}
Let $\frh$ be any nilpotent subalgebra of $\frg$.
Then the restriction of $\met$ to $\frh$ is an invariant bilinear form
on $\frh$.
\end{lem}
\begin{proof}
Let $H\in\frh$.
By nil-invariance of $\met$, the nilpotent part $\ad_{\frg}(H)_{\rm n}$ of
the Jordan decomposition of $\ad_{\frg}(H)$ is skew-symmetric with respect
to $\met$.
Since $\frh$ is a nilpotent subalgebra,
$\ad_{\frh}(H)$ is a nilpotent
operator, and hence $\ad_{\frg}(H)_{\rm n}|_{\frh}=\ad_{\frh}(H)$.
This means the restriction of $\met$ to $\frh$ is an invariant bilinear
form.
\end{proof}

\begin{proof}[Proof of Theorem \ref{thm:BG}]
Suppose that $\frg$ is solvable. Let $H_0$ be a regular element in $\frg$.
Then $\frg_1$ is contained in the nilradical $\frn$ of $\frg$.
Indeed, $\frn\supseteq[\frg,\frg]$ and
$\frg_1=\ad(H_0)\frg_1\subseteq[\frg,\frg]$.

Suppose now that $\frg$ has a nil-invariant symmetric bilinear form $\met$.
In par\-ticular, $\ad(N)$ is skew-symmetric for all $N\in\frn$ and the
restriction of $\met$ to any nilpotent subalgebra is invariant
by Lemma \ref{lem:adh_invariant}.
In particular, the restriction of $\met$ to the Cartan subalgebra
$\frg_0$ is invariant.

Now let $X\in\frg$. Then, for any $N,N'\in\frn$,
\[
\langle\ad(X)X,N\rangle=0,\quad
\langle X,\ad(N)X\rangle=-\langle X,\ad(N)X\rangle=0,
\]
and also
\begin{align*}
\langle \ad(X)N,N'\rangle &= -\langle\ad(N)X,N'\rangle = \langle X,\ad(N)N'\rangle \\
&=
-\langle X,\ad(N')N\rangle = \langle \ad(N')X,N\rangle \\
&=-\langle N,\ad(X)N'\rangle.
\end{align*}
Thus $\ad(X)$ is skew-symmetric for the restriction of
$\met$ to $\RR X+\frn$, and moreover $X\perp[X,\frn]$.

Observe that $\frg_1\subseteq[H_0,\frg_1]\subseteq[H_0,\frn]$, and hence $H_0\perp\frg_1$.
The same holds for all elements $X$ in a non-empty open subset of $\frg_0$
(compare Lemma \ref{lem:regular}), and hence
\[
\frg_0\perp\frg_1.
\]
Altogether, any $X\in\frg_0$ preserves $\met$ on $\frg_1$, and, as stated
before, preserves $\met$ on $\frg_0$, since $\frg_0$ is nilpotent.
Hence $\ad(X)$ is skew-symmetric on $\frg$.
Since $\frg=\frg_0+\frn$, this means $\met$ is an invariant bilinear form
on $\frg$.
\end{proof}




\end{document}